\documentclass[11pt]{amsart}
\usepackage[margin=1.2in]{geometry}
\usepackage{amsmath}
\usepackage{color}
\usepackage{soul}
\usepackage{amsthm}
 \usepackage{lmodern}
\usepackage{amsfonts}
\usepackage{amssymb}
\usepackage{amscd}
\numberwithin{equation}{section}
\newtheorem{theorem}{Theorem}[section]
\newtheorem{conj}[theorem]{Conjecture}
\newtheorem{cor}[theorem]{Corollary}
\newtheorem{proposition}[theorem]{Proposition}
\newtheorem{lemma}[theorem]{Lemma}
\newtheorem{prop}[theorem]{Proposition}
\theoremstyle{definition}
\newtheorem{definition}[theorem]{Definition}
\newtheorem{remark}[theorem]{Remark}

\newcommand\gn{\g^\natural}

\newcommand\half{\tfrac{1}{2}}

\newcommand\be{\beta}

\newcommand\g{\mathfrak g}

\newcommand\ga{\widehat{\mathfrak g}}
\newcommand\h{\mathfrak h}

\newcommand\n{\mathfrak n}

\newcommand\D{\Delta}
\renewcommand\l{\lambda}

\newcommand\Da{\widehat\Delta}
\newcommand\Pia{\widehat\Pi}

\newcommand\Wa{\widehat{W}}
\renewcommand\d{\delta}

\renewcommand\a{\alpha}

\newcommand{\Z}{\mathbb Z}
\renewcommand\th{\theta}

\newcommand\nat{\mathbb N}

\renewcommand\L{\Lambda}

\newcommand\e{\epsilon}
\newcommand\C{\mathbb C}
\newcommand\R{\mathbb R}

\newcommand{\ZZ}{\mathbb{Z}}

\newcommand{\Ws}{W_k^{\min}(\g)}
\newcommand{\Wu}{W^k_{\min}(\g)}
\newcommand{\vac}{{\bf 1}}
\newcommand{\bea}{\begin{eqnarray}}
\newcommand{\eea}{\end{eqnarray}}
\begin{document}
\title{Defining relations for minimal unitary quantum affine W-algebras}
\author[Adamovi\'c, Kac, M\"oseneder, Papi]{Dra{\v z}en~Adamovi\'c}
\author[]{Victor~G. Kac}
\author[]{Pierluigi M\"oseneder Frajria}
\author[]{Paolo  Papi}
\begin{abstract} We prove that any unitary highest weight module over a universal minimal quantum affine $W$-algebra at non-critical level descends to its simple quotient. We find the defining relations of the unitary simple minimal quantum affine $W$-algebras and the list of all their irreducible positive energy modules.  We also classify all irreducible highest weight modules for the  simple affine vertex algebras in the cases when the associated simple minimal $W$-algebra is unitary.
\end{abstract}

\maketitle
\section{Introduction}
Vertex algebras have been playing an increasingly important role in quantum physics (see e.g. \cite{Lemos, Bonetti, Carpi, Cheng, Beem} and references therein). Some of the most relevant to physics among them are unitary vertex  algebras. 

 The problem of unitarity of highest weight representations of infinite-dimensional Lie superalgebras has been a hot topic in Mathematics and Physics in the 1980s and 1990s. Unitary highest weight representations have been classified for many  infinite-dimensional Lie algebras and superalgebras, including affine Lie algebras  \cite{VB}, Virasoro algebra \cite{UV1, GKO0},  $N=1,2,3$ and $4$ superconformal algebras 
 \cite{FQS1,  UV1, ciao,  KR, KW0, KacT, BFK, GKO, Div,  ET1, ET2, ET3, KS, M, FST, A-01,CRW}. All these Lie (super)algebras arise naturally in the context of a special class of vertex algebras, called quantum affine $W$-algebras, which are obtained by quantum Hamiltonian reduction from affine vertex algebras.\par 
Quantum affine $W$-algebras
$W_k(\g,x,f),$ are simple vertex algebras  constructed in \cite{KRW, KW1}, starting from  a datum $(\g,x,f)$ and $k\in\C$. Here $\g=\g_{\bar 0}\oplus \g_{\bar 1}$ is a basic  Lie superalgebra, i.e. $\g$ is simple, its even part $\g_{\bar 0}$ is a reductive Lie algebra and $\g$ carries an even invariant non-degenerate supersymmetric bilinear form $(.|.)$, $x$ is an $ad$--diagonalizable element of $\g_{\bar 0}$ with eigenvalues in $\tfrac{1}{2}\Z$, $f\in\g_{\bar 0}$ is such that $[x,f]=-f$ and the eigenvalues of $ad\,x$ on the centralizer $\g^f$ of $f$ in $\g$ are non-positive, and  $k$ is non-critical, i.e. $k\ne -h^\vee$, where $h^\vee$ is the dual Coxeter number of $\g$. 
Recall that $W_k(\g,x,f)$ is the unique simple quotient  of the universal $W$-algebra, denoted by $W^k(\g,x,f),$ which is  freely strongly generated by elements labeled by a basis of the centralizer of  $f$ in $\g$ \cite{KW1}. 

In the paper \cite{KMP1} we focused on {\it minimal} data $(\g,x,f)$, i.e. the cases when $f$ is an even minimal nilpotent element.  In this case $x$ and $f$ are contained in a (unique) subalgebra $\mathfrak s$, isomorphic to $sl(2)$. The Virasoro, Neveu-Schwarz and $N=2$ algebras cover all minimal $W$-algebras (associated with $\g=sl(2), spo(2|1) ,sl(2|1)$ respectively), such that the $0$-eigenspace  $\g_0$ of $ad\,x$ is abelian: their unitarity, and that of their irreducible highest weight modules,  has been studied in the papers quoted  above. In  \cite{KMP1} we dealt with the cases when $\g_0$ is not abelian, and we  found  all non-critical $k\in\C$ for which the simple minimal $W$-algebras  $W_k(\g,x,f)$, denoted henceforth by $\Ws$, are unitary, along with the (partly conjectural)  classification of  unitary highest weight modules over their universal covers $\Wu$. We call this set of values of $k$ the {\it unitarity range}. 
Recall from \cite{AKMPP} that a level $k$ is said to be {\it collapsing} if $\Ws$ is the simple affine vertex algebra attached to the centralizer $ \g^\natural$ in $\g_0$ of   $\mathfrak s$. If 
$k$ is collapsing, then the unitarity of $\Ws$ reduces to the unitarity of an affine vertex algebras, which is well-understood.  Requiring that the unitary range contains non-collapsing levels  imposes severe restrictions on  the Lie superalgebra  $\g$, namely  $\g$ must be one of the algebras listed in  Table 1. It turns out that the unitarity range (described explicitly in 
Table 1) is precisely the set of levels for which the affine  vertex subalgebra of $\Ws$ generated by $\g^\natural$ is integrable when viewed as a $\ga^\natural$-module. 

The analysis of the unitarity of $\Ws$ developed in \cite{KMP1} led to the more general problem of classifying unitary highest weight modules over $\Ws$; these results,   regarding modules in the Neveu-Schwarz sector, are summarized in Section 2.   

In particular, we gave a mathematically rigorous proofs of the classification 
for the $N=3$ and $N=4$ superconformal algebras, due to  Miki \cite{M} and Eguchi-Taormina \cite{ET1, ET2, ET3}, respectively. 


 \par
It was   conjectured in \cite{KMP1} that actually any unitary highest weight $\Wu$-module descends to $\Ws$. In the present paper we prove this 
conjecture (Theorem \ref{t1} and Corollary \ref{uu}). The proof is based on the classification of irreducible highest weight modules over the simple affine vertex algebras $V_k(\g)$ for these $k$ (Theorem \ref{TA}). We also find generators for the maximal ideal $I^k$ of $\Wu$, so that $\Ws=\Wu/I^k$
(Theorem \ref{C}). Finally, for all non-critical $k\in\C$ in the unitarity range, we classify all irreducible highest weight modules over $\Ws$(Theorem \ref{t2}), using the Zhu algebra method.  
It follows from this classification that all these vertex algebras have infinitely many irreducible modules, and therefore are not rational, unless $k$ is a collapsing level. We also prove that for $\Ws$ all irreducible 
positive energy modules are highest weight modules (Theorem \ref{ipehw}), whereas this statement does not hold for $V_k(\g)$
(Remark 	\ref{non-wh-1}).
 In the final section, we show examples of positive energy modules over $V_k(\g)$ which are not highest weight and  we outline an approach to more general representations of $\Ws$, showing the existence of non-positive energy modules.\par
In the present paper we keep notation and terminology of  \cite{KMP1}; in particular, $\nat$ and $\ZZ_+$ denote the sets of positive and non-negative integer numbers, respectively. The base field is $\C$.

\section{Setup}
\subsection{Minimal $W$-algebras} Let $\g=\g_{\bar 0}\oplus \g_{\bar 1}$ be a simple finite-dimensional  Lie superalgebra over $\C$  with a reductive 
even part $\g_{\bar 0}$ and  a non-degenerate even supersymmetric  invariant bilinear form $(.|.)$. They are classified in \cite{Kacsuper}. Let 
$\mathfrak s=\{e,x,f\}\subset \g_{\bar 0}$ be a {\it minimal} $sl(2)$-triple, i.e.  $[x,e]=e, [x,f]=-f, [e,f]=x$ and the $ad\,x$ gradation of $\g$ has the form 
\begin{equation}\label{1}
\g=\bigoplus_{j\in\tfrac{1}{2}\ZZ}\g_j,\quad\text{where $\g_j=0$ for $|j|>1$, $\g_1=\C e,\,\g_{-1}=\C f$.}
\end{equation} Denote by $\gn$ the centralizer of $\mathfrak s$ in $\g_0$. Let $W^k_{\min}(\g)$ be the associated to $\mathfrak s$ and $k\in\C$ universal  quantum affine $W$-algebra \cite{KW1}. This vertex algebra is strongly generated by 
fields $J^{\{a\}} ,\,a\in \g^{\natural},G^{\{v\}} ,\,v\in \g_{-1/2}, L$ of respective conformal weight $1,3/2,2$ and $\l$-brackets explicitly given in \cite{KW1}.
We normalize the bilinear form $(.|.)$ by the condition $(x|x)=\tfrac{1}{2}$, and denote by $h^\vee$ the corresponding dual Coxeter number. We shall assume throughout the paper that $k\ne -h^\vee$, i.e. $k$ is non-critical. Then the vertex algebra $W^k_{\min}(\g)$ has a unique maximal ideal $I^k$, and we let $W_k^{\min}(\g)=W^k_{\min}(\g)/I^k$ be the corresponding simple vertex algebra.\par
\subsection{Unitary conformal vertex algebras} Let $V$ be a conformal vertex algebra with conformal vector $L=\sum_{n\in\mathbb Z} L_nz^{-n-2}$.  Let $\phi$ be a conjugate linear involution of $V$. A Hermitian form $H(\, . \,\, , \, .\, )$ on $V$ is called $\phi$--{\it invariant} if, for all $a\in V$, one has \cite{KMP}
\begin{equation}
H(v,Y(a,z)u)=H(Y(A(z)a,z^{-1})v,u),\quad u,v\in V.
\end{equation}
Here the linear map $A(z):V\to V((z))$ is defined by 
\begin{equation}\label{12}
A(z)=e^{zL_1}z^{-2L_0}g,
\end{equation}
where 
\begin{equation}\label{113}
g(a)=(-1)^{L_0+2L_0^2}\phi(a),\quad a\in V.
\end{equation}
We say that $V$ is {\it unitary} if there is a conjugate linear involution $\phi$ of $V$ and a $\phi$-invariant Hermitian form on $V$ which is positive definite.
\subsection{Unitary minimal $W$-algebras and unitary highest weight modules.}In classifying the unitary minimal $W$-algebras $\Ws$, one first considers the levels $k$ when $\Ws$ is either $\C$ or an affine vertex algebra. Such levels are called {\sl collapsing levels}.
The triples $(\g,\mathfrak s,k)$ such that $\Ws=\C$ are described in Proposition 3.4 of \cite{AKMPP}. Corollary 7.12 of \cite{KMP} provides the list of triples $(\g,\mathfrak s,k)$ such that $\Ws$ is an affine unitary vertex algebra. 

Turning to the non collapsing levels, 
it is proven in \cite[Proposition 7.9]{KMP} that, if  $W^{\min}_k(\g)$ is  unitary and $k$ is not a collapsing level, then the parity of $\g$ is compatible with the $ad\,x$--gradation, i.e. the parity of the whole subspace $\g_j$ is $2j\mod 2$.
It follows from \cite{KRW}, \cite{KW1} that for each basic simple Lie superalgebra $\g$ there is at most one minimal datum $(\g,\mathfrak s)$ that is compatible with parity, and  the  complete list of the  $\g$ which admit such a datum is as follows:
\begin{align}\label{ssuper}
\begin{split}
sl(2|m)\text{ for $m\ge3$},\quad psl(2|2),\quad  spo(2|m)\text{ for $m\geq 0$,}\\
osp(4|m)\text{ for $m>2$ even},\quad D(2,1;a),\quad  F(4),\quad G(3).
\end{split}
\end{align}
\begin{remark}\label{d21a}Recall from \cite{Kacsuper} that in  the case of $D(2,1 ; a), $ the parameter $a$ ranges over $\C\setminus\{0,-1\}$ (if $a=0,-1$ the superalgebra is not simple); moreover the symmetric group $\mathcal S_3$ generated by the tranformations $a\mapsto -1-a, a\mapsto 1/a$ induces Lie superalgebra isomorphisms.\end{remark}\par 
  The even part $\g_{\bar 0}$ of $\g$ in these cases is $\g^\natural\oplus \mathfrak s$ with $\g^\natural$ a reductive Lie algebra.
In Propositions 7.1. and 7.2 of \cite{KMP1} it is proven that  the conjugate linear involutions of $\Ws$ that fix the Virasoro vector $L$ are in one-to-one correspondence with the conjugate linear involutions $\phi$ of $\g$  that fix pointwise the triple $\{e,x,f\}$. It is easy to see that, in order to have $\Ws$ unitary, $\phi_{|\g^\natural}$ must be the conjugation corresponding to a compact real form. Such a conjugate linear involution is called an {\sl almost compact involution} and Proposition 3.2 of \cite{KMP1} shows that an almost compact involution $\phi_{ac}$ exists in all the cases listed in \eqref{ssuper}. 
Since the cases $\g=spo(2|m),\,m=0,1,$ and $2$, correspond to the well understood cases of the simple Virasoro, Neveu-Schwarz and $N=2$ vertex algebras, we exclude these 
$\g$ from consideration. For the other cases, Theorem 1.4 of \cite{KMP1} lists all the pairs
 $(\g,k)$ for which the $\phi_{ac}$-invariant hermitian form on $\Ws$  is positive definite, making $\Ws$ a unitary vertex algebra. We say that $k$  belongs to the unitary range if  $k$ is not collapsing and $W_k^{\min}(\g)$ is a unitary vertex algebra.
 
  We have \cite{KMP1}
\begin{equation}\label{2}
\gn=\bigoplus_{i\in S}\gn_i\text{ and } \g_0=\gn\oplus \C x,\end{equation}
where $\gn_i$ are simple summands of $\gn$, and either $S=\{1\}$ or $S=\{1,2\}$; the latter happens  only for $D(2,1;\tfrac{m}{n})$. Let $\h^\natural $ be a Cartan subalgebra of $\g^\natural$, then $\h=\h^\natural +\C x$ is a Cartan subalgebra of $\g_0$ and of  $\g$. Let $\theta\in\h^*$ be the root of $e$, so that $(\theta|\theta)=2$, and denote by $\theta_i\in (\h^\natural)^*$ the highest root of $\gn_i$.
Denote by $M_i(k)$ the levels of the affine Lie algebras $\widehat\g_i^\natural$ in $W^k_{\min}(\g)$. Then we have for $i\in S$ \cite{KMP1}
\begin{align}\label{MIK}M_i(k)&=
\frac{2k}{(\theta_i|\theta_i)}+\chi_i,\\
\label{chii}\chi_i&=-\xi(\theta_i^\vee),\end{align} 
where $\xi\in(\h^\natural)^*$ is  the highest weight of the $\g^\natural$-module $\g_{-1/2}$ (this module is irreducible, except for $\g=psl(2|2)$, when its two  irreducible components
turn out to be isomorphic). In \cite{KMP1} it is shown that  $\chi_i=-2$ when $\g= spo(2|3)$,  and  $\chi_i=-1$ otherwise.\par
 In Table  1 we describe the unitary range in all cases, along with  the subalgebras $\g^\natural$, the numbers $M_i(k)$, and  the dual Coxeter numbers $h^\vee$.
\begin{table}[h]\label{Table1}
\begin{tabular}{c | c| c |c | c | c |c}
$\g$&
$psl(2|2)$&
$spo(2|3)$&
$spo(2|m), m>4$&
$D(2,1;\tfrac{m}{n}),\,m,n\in\nat$&
$F(4)$&
$G(3)$\\\hline
$-k$&$\mathbb N+1$&$\frac{1}{4}(\mathbb N+2)$&$\frac{1}{2}(\mathbb N+1)$&$\frac{mn}{m+n}\nat,\ (m,n)=1$
&$\frac{2}{3}(\mathbb N+1)$&$\frac{3}{4}(\mathbb N+1)$
\\\hline
$\g^\natural$&$sl(2)$&$sl(2)$&$so(m)$&$sl(2)\oplus sl(2)$&$so(7)$&$G_2$
\\\hline
$-M_i(k)$&$k+1$&$4k+2$&$2k+1$&$\frac{m+n}{n}k+1, \frac{m+n}{m}k+1$ &$\frac{3}{2} k+1$ & $\frac{4}{3} k+1$
\\\hline
$h^\vee$&$0$&$\tfrac{1}{2}$&$2-\tfrac{1}{2}m$&$0$ &$-2$ & $ -\tfrac{3}{2}$
\end{tabular}
\vskip8pt
\caption{Unitarity range, $\g^\natural, M_i(k)$, and $h^\vee$}
\end{table}

We also make a specific choice for a set of simple roots $\Pi$ of $\g$ in each case.
In Table 2 we list our choice of $\Pi=\{\a_1,\ldots\}$ of $\g$, ordered from left to right, as well as the invariant bilinear form 
$(.|.)$ on $\h^*$, and the  root $\theta$. 
\begin{table}[h]\label{Table2}
{\scriptsize
\begin{tabular}{c | c| c |c }
$\g$&$\Pi$&\text{ $(. |.)$} & $\theta$\\
\hline
$psl(2|2)$&
$\{\e_1-\d_1,\d_1-\d_2,\d_2-\e_2\}$&$(\e_i|\e_j)=\d_{i,j}=-(\d_i|\d_j)$ & $\e_1-\e_2$\\
& &$(\e_i|\d_j)=0$
\\\hline
$spo(2|2m+1), m\ge1$&
$\{\d_1-\e_1,\e_1-\e_2,\ldots,\e_{m-1}-\e_m,\e_m\}$&$(\e_i|\e_j)=-\tfrac{1}{2}\d_{i,j}, (\d_1|\d_1)=\tfrac{1}{2},\,(\e_i|\d_1)=0$& $2\d_1$
\\\hline
$spo(2|2m), m\ge 3$&
$\{\d_1-\e_1,\e_1-\e_2,\ldots,\e_{m-1}-\e_m,\e_{m-1}+\e_m\}$&$(\e_i|\e_j)=-\tfrac{1}{2}\d_{i,j}, (\d_1|\d_1)=\tfrac{1}{2},\,(\e_i|\d_1)=0$& $2\d_1$
 \\\hline
$D(2,1;a)$&$\{\e_1-\e_2-\e_3, 2\e_2, 2\e_3\}$& $(\e_1|\e_1)=\frac{1}{2}, (\e_2|\e_2)=\frac{-1}{2(1+a)},  (\e_3|\e_3)=\frac{-a}{2(1+a)}$ & $2\e_1$\\
 && $(\e_1|\e_2)=(\e_1|\e_3)=(\e_2|\e_3)=0$ \\\hline
$F(4)$&$\{\tfrac{1}{2}(\d_1-\e_1-\e_2-\e_3), \e_3,\e_2-\e_3,\e_1-\e_2\}$&$(\e_i|\e_j)=-\tfrac{2}{3}\d_{i,j}, (\d_1|\d_1)=2$& $\d_1$\\
& &$(\e_i|\d_1)=0$
\\\hline
$G(3)$&$\{\d_1+\e_3, \e_1,\e_2-\e_1\}$&$(\e_i|\e_j)=\frac{1-3\d_{i,j}}{4}, (\d_1|\d_1)=\frac{1}{2}$& $2\d_1$\\
& &$(\e_i|\d_1)=0,\ \e_1+\e_2+\e_3=0$
\end{tabular}
}

\caption{Simple roots $\Pi$, invariant form $(.|.)$, and the highest root $\theta$ of $\g$}

\end{table}

Let $\D^\natural$ be the set of roots of $(\g^\natural,\h^\natural)$. We made our choice of $\Pi$ so that $\Pi^\natural=\Pi\cap \D^\natural$ is a set of simple roots for $\g^\natural$. Write
$\g^\natural=\n^\natural_-\oplus\h^\natural\oplus\n^\natural_+$ for the triangular decomposition of $\g^\natural$ corresponding to choosing $\Pi^\natural$ as a set of simple roots. Note that $\a_1$ is an isotropic root, and that we have
\begin{equation}\label{a1}
(\a_1)_{|\h^\natural}=-\xi.
\end{equation}
We parametrize the highest weight modules for $W^k_{\min}(\g)$ following Section 7 of \cite{KW1}.
 For  $\nu\in (\h^\natural)^*$ and $\ell_0\in\C$, let $L^W(\nu,\ell_0)$ denote the irreducible highest weight  $W^k_{\min}(\g)$--module with highest weight $(\nu,\ell_0)$ and highest weight vector $v_{\nu,\ell_0}$. This means that one has
\begin{align*}
&J^{\{h\}}_0v_{\nu,\ell_0}=\nu(h)v_{\nu, \ell_0} \text{ for } h\in\h^\natural,\quad L_0v_{\nu,\ell_0}=\ell_0v_{\nu,\ell_0},\\
&J^{\{u\}}_nv_{\nu,\ell_0}=G^{\{v\}}_nv_{\nu,\ell_0}=L_nv_{\nu,\ell_0}=0\text{ for $n>0$, $u\in \g^\natural$},\ v\in\g_{-1/2}, \\
&J^{\{u\}}_0v_{\nu,\ell_0}=0\text{ for } u\in\n^\natural_{+}.
\end{align*}
Let $P^+\subset (\h^\natural)^*$ be the set of dominant integral weights  for $\g^\natural$ and let
\begin{equation}\label{p+k}P^{+}_k=\left\{\nu\in P^+\mid  \nu(\theta^\vee_i)\le M_i(k)\text{ for all $i\ge 1$}\right\}.\end{equation}

\begin{definition} An element $\nu\in P^+_k$ is called an {\it extremal weight} if 
$\nu+\xi$ doesn't lie in $P^+_k$. Equivalently, $\nu$ is extremal if
\begin{equation}\label{extr}\nu(\theta_i^\vee)>M_i(k)+\chi_i\text{ for some $i\in S$}.\end{equation}
\end{definition}
 For $\nu\in P^+$, introduce the following number \cite{KMP1}
\begin{equation}\label{Aknu}
 A(k,\nu)=\frac{(\nu|\nu+2\rho^\natural)}{2(k+h^\vee)}+\frac{(\xi|\nu)}{k+h^\vee}((\xi|\nu)-k-1),
 \end{equation}
where $2\rho^\natural$ is the sum of positive roots of $\g^\natural$. The main result of \cite{KMP1} is the following theorem.
\begin{theorem}\label{nec} Let  $L^W(\nu,\ell_0)$ be  an irreducible highest weight $W^k_{\min}(\g)$--module for  $\g= psl(2|2),$ $ spo(2|m)$ with $m\ge 3, \, D(2,1;a),\, F(4)$ or $G(3)$.
\begin{enumerate}
\item This module can be unitary  only if the following conditions hold:
\begin{enumerate}
\item $M_i(k)$ are non-negative integers,
\item $\nu(\theta_i^\vee) \leq M_i(k)$ for all $i$,
\item
\begin{equation}\label{eh}
\ell_0\ge A(k,\nu),
\end{equation}
 and equality holds in \eqref{eh} if $\nu(\theta^\vee_i)>M_i(k)+\chi_i$ for $i=1$ or $2$.
 \end{enumerate}
\item This module is unitary if the  following conditions hold:
\begin{enumerate}
\item $M_i(k)+\chi_i\in\mathbb Z_+$ for all $i$,
\item  $\nu(\theta^\vee_i) \leq M_i(k)+\chi_i$ for all $i$ (i.e. $\nu$ is not extremal),
\item  inequality \eqref{eh} holds.
\end{enumerate}\end{enumerate}
\end{theorem}
Applying this theorem to $L^W(0,0)=\Ws$ one recovers the unitarity range displayed in Table 1.
\subsection{A technical result} Let $\D$ be the set of roots of $\g$ and let $\Pia=\{\a_0=\d-\theta,\a_1,\ldots\}$ be a set of simple roots of the affine Lie superalgebra $\ga$.
For an isotropic root $\beta\in \Pia$, we denote by  $r_{\be}(\Pia)$ the set of simple  roots in the set $\Da$ of roots of $\ga$ obtained by the corresponding odd reflection. 
We denote by $x_\a$ a root vector of $\ga$, attached to $\a\in \Da$, and by $w.$ the shifted action of $\Wa$: $w.\l= w(\l+\widehat \rho)-\widehat \rho$.
 
  \begin{lemma}\cite[Lemma 11.3]{KMP1} \label{oddref} Let $\widehat \Pi'$ be a set of simple roots for $\Da$. Let $M$ be a $\ga$--module and assume that $m\in M$ is a singular  vector with respect to $\widehat \Pi'$. If $\a\in\widehat \Pi'$ is an isotropic root and  $x_{-\a}m\ne0$, then $x_{-\a}m$ is a singular vector with respect to $r_{\a}(\widehat \Pi')$.
 \end{lemma}
\vskip5pt




\section{ Classification of irreducible highest weight  representations of  $V_k(\g)$ in the unitary range}

 Denote by $V^k(\g)$ and $V_k(\g)$ the universal and simple affine vertex algebras of level $k$ associated to $\g$.  Denote by $L(\l)$ the irreducible 
highest weight $V^k(\g)$-module of highest weight $\l$.  For $h\in \C$ and $\nu\in (\h^\natural)^*$, set  
\begin{equation}\label{nuh}\widehat \nu_h=k\L_0+h\theta+\nu.\end{equation}
Note that every highest weight module for  $V^k(\g)$ has highest weight $\widehat \nu_h$ for some $\nu\in   (\h^\natural)^*$ and $h\in\C$.
 \begin{theorem}\label{TA} Let $k$ be in the unitary range. Then, up to isomorphism, the   irreducible highest weight   $V_k(\g)$-modules  are as follows:
 \begin{enumerate}
 \item $L(\widehat\nu_h)$ with $\nu\in P^+_k$ non-extremal and  $h\in\C$ arbitrary;
 \item $L(\widehat\nu_{h})$ with $\nu$ extremal and  $h$ from the set $E_{k,\nu}=\{(\xi|\nu),k+1-(\xi|\nu)\}$.
 \end{enumerate}
 \end{theorem}
 \begin{lemma}\label{fund} Let $k$ be in the unitary range.  Then
 $V_k(\g)$ is integrable for $\ga^\natural$.
\end{lemma}
\begin{proof}[Proof of Lemma \ref{fund}]It is enough to check that $(x_{\theta_i})_{(-1)}^N\vac=0$ for some $N\in\nat$ and all $i$. 
Using Table 2 one easily checks that $\theta-\a_1$ is a root of $\g$. Since $(x_{\a_1})_{(0)}\vac=0$, $\vac$ is a singular vector in $V^k(\g)$ also for $r_{\a_1}(\widehat\Pi)$. Since $(x_{\theta-\a_1})_{(-1)}\vac$ is nonzero in $V^k(\g)$, it follows from Lemma \ref{oddref} that $(x_{\theta-\a_1})_{(-1)}\vac$ is a  singular vector in $V^k(\g)$ for $\widehat\Pi'=r_{\a_0+\a_1}r_{\a_1}(\widehat\Pi)$. Note that, since $k\ne 0$, \begin{align*}(x_{-\theta+\a_1})_{(1)} (x_{\theta-\a_1})_{(-1)}\vac&=
[(x_{-\theta+\a_1})_{(1)}, (x_{\theta-\a_1})_{(-1)}]\vac=((h_{\th-\a_1})_{(0)}+k(x_{-\theta+\a_1}|x_{\theta-\a_1}))\vac\\&=k(x_{-\theta+\a_1}|x_{\theta-\a_1})\vac = \gamma \vac,\quad\gamma\ne 0.\end{align*}
In particular, $(x_{\theta-\a_1})_{(-1)}\vac$  generates $V^k(\g)$.

 Let $\L'=k\L_0-\a_0-\a_1=k\L_0-\d+\theta-\a_1$ be the weight of $(x_{\theta-\a_1})_{(-1)}\vac$. By a case-wise verification one shows that the roots  $\eta_i:=\d-\theta_i$ are   in $\widehat\Pi'$ and,  by \eqref{MIK} and \eqref{chii},
$$
(\L'|\eta_i^\vee)=\frac{2}{(\theta_i|\theta_i)}k+(\a_1|\theta_i^\vee)=M_i(k)-\chi_i-(\xi|\theta_i^\vee)
=M_i(k).
$$
Hence the vector $(x_{\theta_i})_{(-1)}^{M_i(k)+1}(x_{\theta-\a_1})_{(-1)}\vac$ is singular  in $V^k(\g)$ for $\widehat\Pi'$. Since the vector $(x_{\theta-\a_1})_{(-1)}\vac$ generates $V^k(\g)$, it follows that $V_k(\g)$ is $\widehat \g^\natural$--integrable.
\end{proof}
\begin{proof}[Proof of Theorem \ref{TA}] 
 In \cite[Sections 4 and 6]{GK2} the characters of highest weight $\ga$-modules with highest weight $k\L_0$ have been computed using only their integrability with respect to $\ga^\natural$, which implies that such modules are irreducible. It was deduced from this in  \cite[Theorem 5.3.1]{GS} that, if  $V_k(\g)$  is integrable as a $\ga^\natural$-module,  then  the $V^k(\g)$--modules,  which are integrable over $\ga^\natural$,
descend to $V_k(\g)$. \par
By  Lemma \ref{fund}, we are left with proving  that  the  modules listed in (1), (2)  are  $\ga^\natural$--integrable. Let $v$ be a highest weight vector for $L(\widehat\nu_h)$.

{\sl Case (1):} $\nu$ is not extremal, $h\in\C$. If $(\widehat \nu_{h}|\a_1)\ne 0$ and $(\widehat \nu_{ h}-\a_1|\a_0+\a_1)\ne 0$, then $(x_{\theta-\a_1})_{(-1)}(x_{-\a_1})_{(0)}v$ is a singular vector with respect  to the set of simple roots $r_{\a_0+\a_1}r_{\a_1}(\Pia)$. 
 Moreover, by \eqref{MIK},\eqref{chii}, and \eqref{a1} we have for $i\in S$
\begin{align}
m_i:= (\widehat\nu_{ h}-2\a_1-\a_0|\eta_i^\vee)&=M_i(k)-\chi_i+2(\a_1|\theta_i^\vee)-(\nu|\theta_i^\vee)\notag\\
 &=M_i(k)-\chi_i-2(\xi|\theta_i^\vee)-(\nu|\theta_i^\vee)\notag\\
 &=M_i(k)+\chi_i-(\nu|\theta_i^\vee)\in\ZZ_+\label{mik}
\end{align}
by \eqref{extr}. Since $\eta_i\in r_{\a_0+\a_1}r_{\a_1}(\Pia)$, we see that,  in $L(\widehat\nu_{ h})$,
$$(x_{\theta_i})_{(-1)}^{m_i+1}(x_{\theta-\a_1})_{(-1)}(x_{-\a_1})_{(0)}v=0,
$$
hence $L(\widehat \nu_{ h})$ is $\ga^\natural$--integrable.

If $(\widehat \nu_{ h}|\a_1)= 0$, since $(\widehat  \nu_{ h}|\a_0+\a_1)=(\widehat  \nu_{ h}|\a_0)\ne0$, we have that $(x_{\theta-\a_1})_{(-1)}v$ is a singular vector for $r_{\a_0+\a_1}r_{\a_1}(\widehat\Pi)$.
Moreover,
\begin{align}\label{ne1}
(\widehat\nu_{ h}-\a_1-\a_0|\eta_i^\vee)=M_i(k)-(\nu|\theta_i^\vee)\in\ZZ_+
\end{align}
since $\nu\in P^+_k$. We can therefore conclude as above that  $L(\widehat \nu_{ h})$ is $\ga^\natural$--integrable.

Finally, if $(\widehat \nu_{ h}|\a_1)\ne 0$ and $(\widehat \nu_{ h}-\a_1|\a_0+\a_1)= 0$, then $(x_{-\a_1})_{(0)}v$ is a singular vector for $r_{\a_0+\a_1}r_{\a_1}(\widehat\Pi)$.
Moreover,
\begin{align}\label{ne2}
(\widehat\nu_{ h}-\a_1|\eta_i^\vee)=M_i(k)-(\nu|\theta_i^\vee)\in\ZZ_+,
\end{align}
and we conclude as in the previous case.\par
{\sl Case (2):} $\nu$ is extremal, $ h\in E_{k,\nu}$. 
If $ h=(\xi|\nu)$, then 
 $(\widehat\nu_{ h}|\a_1)=0$.  We  can then conclude as in the non-extremal case. 
 If $ h=k+1-(\xi|\nu)$,
we may assume that $(\widehat\nu_{h}|\a_1)\ne 0$.  Then, since  $(\theta|\alpha_1)=1$, we have
\begin{align*}
(\widehat\nu_{ h}-\a_1|\a_0+\a_1)&=(k\L_0+(k+1-(\xi|\nu))\theta+\nu-\a_1|\a_0+\a_1)\\
&=k-2(k+1-(\xi|\nu))+1+(k+1-(\xi|\nu))+(\nu|\a_1)=0.
\end{align*}
We can therefore conclude as in the non-extremal case.\par
We are left with proving that any irreducible highest weight module for $V_k(\g)$ is  of the form (1) or (2). Let $L(\widehat \nu_h)$ be an  irreducible highest weight $V_k(\g)$-module. We prove that,  necessarily,  $\nu\in P^+_k$. Indeed, the action of $\g^\natural$ on a highest weight vector $v$ should be locally finite, so that $\nu\in P^+$. 
If $h\notin E_{k,\nu}$ then 
$(\widehat \nu_{h}|\a_1)\ne 0$ and $(\widehat \nu_{ h}-\a_1|\a_0+\a_1)\ne 0$. 
It follows that  $(x_{\theta-\a_1})_{(-1)}(x_{-\a_1})_{(0)}v$ is a highest weight vector with respect to the set of simple roots $r_{\a_0+\a_1}r_{\a_1}(\Pia)$ of 
highest weight $\L'=\widehat\nu_h-\a_0-2\a_1$. If $L(\widehat \nu_h)$ is integrable with respect to $\ga^\natural$, 
the computation done in \eqref{mik} shows that 
$m_i$ should be a non-negative integer, hence
$$(\nu|\theta_i^\vee)\leq M_i(k)+\chi_i\leq M_i(k).$$
It follows that $\nu\in P^+_k$ and it is not extremal, i.e. $ L(\widehat \nu_h)$ is of type (1).\par
If $h=(\xi|\nu)$ (resp. $h=k+1-(\xi|\nu)$), then $(\widehat \nu_{ h}|\a_1)= 0$ (resp.  $(\widehat \nu_{ h}|\a_0+\a_1)= 0$ ) and as in \eqref{ne1} (resp. \eqref{ne2})  we get that $\nu\in P^+_k$. In particular, $ L(\widehat \nu_h)$ is of type (1) if $\nu$ is not extremal and of type (2) if $\nu$ is extremal.
 \end{proof}

 \begin{remark} \label{non-wh-1} All modules listed in Theorem \ref{TA} are of positive energy (the definition is recalled in Section \ref{Section 6}),  but there might exist positive energy $V_k(\g)$--modules outside of this list. In \S\ 7.1 we present arguments for this  claim in the case 
 $\g = spo(2 \vert 3)$, $k =- m/4$, and $m\ge 4$ even.\par One of the referees pointed out that in \cite[5.6.4]{GS}  the authors give an example of an irreducible positive energy module which is not a highest weight module.  Similar examples,  with $\g$ as in Table 1  and $k$ in the unitarity range,  are given in  \cite[5.6.6]{GS} (note that in \cite{GS} a  normalization of the bilinear form different from ours is used).

Let us explain what happens in our situation. If $M$ is  a positive energy module for a vertex algebra $V$, we set  
$M_{top}=Zhu(M)$, where $Zhu( - )$ is the Zhu functor between  positive energy  modules and  modules for the  Zhu algebra
$A(V)$.  We  start with the $V_k(\g)$--module
$L(\widehat\nu_h)$ with $\nu\in P^+_k$ non-extremal and  $h\in\C$ arbitrary, from Theorem \ref{TA}. Then $L(\widehat\nu_h)_{top} $ contains an irreducible   $\g_{\bar 0}$--submodule $E= U(\g_{\bar 0}). v$, where $v$ is the  highest weight vector of  $L(\widehat\nu_h)$.
Next we consider the   Kac module $E_1$, which is a suitable quotient of  $Ind_{\g_{\bar 0}} ^{\g} E$. As mentioned in \cite{GS}, if all $M_i(k)$ are  large enough, we get that $E_1$ is a module for Zhu's algebra $A(V_k(\g))$.  Using Zhu's functor we conclude the following:
\begin{itemize}
\item there exists a positive energy  $V_k(\g)$--module  $\widetilde {L}(\widehat\nu_h)$ such that $\widetilde {L}(\widehat\nu_h)_{top} = E_1$.
\item the module $\widetilde {L}(\widehat\nu_h)$ is indecomposable and  has an irreducible subquotient isomorphic to $L(\widehat\nu_h)$.

\end{itemize}
Hence  non-highest weight positive energy modules for $V_k(\g)$ exist. But still these modules are weight modules with finite-dimensional weight spaces.

Let us mention one important consequence. Take $h \in {\C}$ such that $E$ is finite-dimensional, irreducible module and assume that $M_i(k)$ are large enough. Then $E_1$ is an indecomposable  finite-dimensional module for Zhu's algebra  $A(V_k(\g))$, and therefore the category $KL_k$ \cite{AMP22} is not semisimple. 
In the  case $\g = spo(2 \vert 3)$, one can show that the category $KL_{-m/4}(\g)$ is not semisimple for $ m \ge 4$.

\end{remark}

\section{Explicit description of the maximal ideal of $W^k_{\min}(\g)$}

Let, as before, $\Pi=\{\a_1,\ldots\}$ be the set of simple roots for $\g$ given in Table 2, and  $I^k$ be the maximal ideal of $W^k_{\min}(\g)$.  Also   denote by $J^k$ the maximal ideal of $V^k(\g)$. Set 
\begin{equation}\label{vvi}
v_i= (J^{\{x_{\theta_i}\}}_{(-1)})^{M_i(k)+1}\vac,\quad i \in S.
\end{equation}
 If an irreducible highest weight   $W^k_{\min}(\g)$--module $L^W(\nu,\ell_0)$ is   unitary, then, 
restricted to the affine subalgebra $V^{\beta_k}(\g^\natural)$ (see \cite[(7.4), (7.5)]{KMP1} for the definition of this subalgebra), it is unitary, hence a direct sum of irreducible
integrable highest  weight  $\widehat \g^\natural$--modules of levels $M_i(k),\  i\in S$.
But it is well-known \cite{FZ}, \cite{KWa} that all these modules descend to  the simple affine vertex algebra $V_{\beta_k}(\g^\natural)$, and  are annihilated 
by the elements $v_i$.
In particular, applying this argument to $\Ws=L^W(0,0)$, we deduce that 
\begin{equation}\label{vinik}v_i\in I^k.\end{equation}
\begin{theorem}\label{C}  Let $k$ be in the unitary range. The maximal ideal  $I^k$ is generated by the singular vector 
\begin{equation}\label{tv1}\widetilde v_1=
 (J^{\{x_{\theta_1}\}}_{(-1)} )^{M_1(k)-1}  G^{\{x_{-\a_1}\}}_{(-1)}\vac\end{equation} if $\g= spo(2|3)$ and  by the singular vectors 
$v_i,\ i\in S$ (cf. \eqref{vvi}) in the other cases.
\end{theorem}
We split the proof according to whether $\g\ne spo(2|3)$ or $\g= spo(2|3)$.
\subsection{$\g\ne spo(2|3)$} In this case $\chi_i=-1,\, i\in S$.

\begin{lemma}\label{xthetaM}Set $\eta_i=\d-\theta_i,\, i\in S$.  One has 
\begin{align}&\dim V^k(\g)_{k\L_0-j\eta_i}=1\text{ for all $j\ge 0$}, \label{nil}\\
\label{inl}&\dim V_k(\g)_{k\L_0-j\eta_i}=1 \text{ if and only if } j\le s, \end{align}
where 
\begin{equation}\label{s} s=(k\L_0|\eta_i^\vee).\end{equation}
In particular, for each $M\in\Z_{\ge 0}$, the vector  $(x_{-\theta+\a_1})_{(1)}(x_{\theta_i})_{(-1)}^{M}(x_{\theta-\a_1})_{(-1)}\vac$ is a multiple of $(x_{\theta_i})_{(-1)}^{M}\vac$. \end{lemma}
\begin{proof}
Any basic Lie superalgebra $\g$ different form $spo(2|3)$ admits a set of simple roots $\Pi''$ with the property that $\theta_i$ is the maximal root, hence the root $\eta_i=\d-\theta_i$ lies in $\widehat\Pi''$. Since $V^k(\g)$ is the vacuum module, $x_\a\vac=0$  for any root $\a\in\D$, in particular $x_\a\vac=0$ for any $\a\in \Pi''$. Since $(x_{-\theta_i})_{(1)}\vac=0$, $\vac$   is a singular vector also for $\widehat\Pi''$.
Since $\eta_i$ is simple in $\widehat\Pi''$, it is clear that  $V^k(\g)_{k\L_0-j\eta_i}\subset \C (x_{\theta_i})_{(-1)}^j\vac$, and indeed $\dim V^k(\g)_{k\L_0-j\eta_i}=1$, since $(x_{\theta_i})_{(-1)}^j\vac\ne 0$ in $V^k(\g)$. This proves \eqref{nil}. By Lemma \ref{fund}, $V_k(\g)$ is integrable for $\ga^\natural$, in particular, 
 $(x_{\theta_i})_{(-1)}^{j}\vac=0$ if and only if  $j>(k\L_0|\eta_i^\vee)$, hence  \eqref{inl} holds.
 The final claim follows from \eqref{nil} by comparing weights.
\end{proof}
Introduce the following vectors in $V^k(\g)$, where $i\in S$:
\begin{align}
w_i&=(x_{\theta_i})_{(-1)}^{M_i(k)+1}(x_{\theta-\a_1})_{(-1)}\vac,\\
s_i&=(x_{\a_1})_{(0)}w_i,\\
u_i&=(x_{-\theta+\a_1})_{(1)}w_i.
\end{align}
\begin{prop}\label{singprime} The vectors $s_i$ 
are singular in the universal affine vertex algebra $V^k(\g)$ and generate $J^k$.
\end{prop}
\begin{proof}Recall that $\a_0=\d-\theta$. Since $(x_{\a_1})_{(0)}\vac=0$, $\vac$ is singular also for $r_{\a_1}(\widehat\Pi)$. Since $(x_{\theta-\a_1})_{(-1)}\vac$ is nonzero, it follows from Lemma \ref{oddref} that $(x_{\theta-\a_1})_{(-1)}\vac$ is singular in $V^k(\g)$ for $\widehat\Pi'=r_{\a_0+\a_1}r_{\a_1}(\widehat\Pi)$. Let $\L'=k\L_0-\d+\theta-\a_1$ be the weight of $(x_{\theta-\a_1})_{(-1)}\vac$. \par Since $\d-\theta_i$ is  in $\widehat\Pi'$ and,  by \eqref{MIK} and \eqref{chii},
$$
(\L'|(\d-\theta_i)^\vee)=\frac{2}{(\theta_i|\theta_i)}k+(\a_1|\theta_i^\vee)=M_i(k)-\chi_i-(\xi|\theta_i^\vee)
=M_i(k), 
$$
 we see that $w_i$ is singular for $\widehat\Pi'$.
Since 
$V^k(\g)/\langle w_i\rangle$
is  integrable with respect to $\ga^\natural$, it is irreducible, because the computation of its character formula in \cite{GK2} did not use irreducibility, but only integrability. Hence the vectors $w_i$
 generate the maximal proper ideal of $V^k(\g)$.

The weight of $u_i$ is $k\L_0 -s'\eta_i$, where
$$
s':=1+(\L'|\eta_i^\vee)=1+(k \L_0|\eta_i^\vee)-(\a_0+\a_1|\eta_i^\vee)=1+s+(\a_1,\theta_i^\vee),
$$
and $s$ is defined in \eqref{s}.
 Since $(\a_1|\theta_i^\vee)\le -1$, we have $s'\le (k \L_0|\eta_i^\vee)$, so $(J^k)_{k\L_0-s'\eta_i} =0$ by Lemma \ref{xthetaM} (2). Since $w_i\in J^k$, we have  $u_i\in J^k$, so $u_i=0$.

We now observe that the fact that $u_i=0$ implies that $w_i$ is a singular vector for $r_{-\a_0-\a_1}(\widehat\Pi')$ $=r_{\a_1}(\widehat\Pi)$. Indeed, if  $\be\ne\a_0+\a_1$ is a simple root for $r_{-\a_0-\a_1}(\widehat\Pi'),$ then it is a positive root for $\widehat\Pi'$, hence, since $w_i$ is singular for $\widehat\Pi'$, $(x_\be)_{(0)}(x_{\theta_i})_{(-1)}^{M_i(k)+1}(x_{\theta-\a_1})_{(-1)}\vac=0$. Observe that, since $\a_0=\d-\theta$,  $u_i=0$ simply means that 
$$
(x_{-\theta+\a_1})_{(1)}w_i=0.
$$
 Having shown that $w_i$ are singular for $r_{\a_1}(\widehat\Pi)$, it follows that the $s_i$ are either zero or  singular vectors for $\widehat\Pi$.
 
Let $\L'''_i=k\L_0-(M_i(k)+1) (\d-\theta_i)-\d+\theta-\a_1$ be the weight of $w_i$.
 Since $k\le 0$ (see Table 1), we have
 $$
 (\L'''_i|\a_1)=((M_i(k)+1)\theta_i|\a_1)+1=-k+1>0,
 $$
 so,
$$(x_{-\a_1})_{(0)}s_i=c(-k+1)w_i-(x_{\a_1})_{(0)}(x_{-\a_1})_{(0)}w_i, \text{ where $c\in \C$}.
$$

Since $w_i$ are  singular vectors  for $r_{\a_1}(\widehat \Pi)$, we see that 
$$(x_{-\a_1})_{(0)}w_i=0.
$$
 It follows that $(x_{-\a_1})_{(0)}s_i$
are nonzero multiples of $w_i$,  hence the $s_i$ generate the maximal proper ideal of $V^k(\g)$, since the $w_i$ do.
\end{proof}

\noindent {\it Proof of Theorem \ref{C} for $\g\ne spo(2|3)$.} 
Since $L(k\Lambda_0) = V^k(\g)/J^k$ is irreducible, $H_0$ is exact on category $\mathcal O$ and $H_0(L(k\L_0))\ne 0$, by \cite{Araduke}, we have that the $\Wu$-module  
$H_0(V^k(\g))/H_0(J^k)$ irreducible, so $I^k=H_0(J^k)$. Assume first that $|S|=1$. From Proposition \ref{singprime}, we deduce that $I^k=H_0(J^k)$ is a highest weight module.  Due to  \cite[(2.19)]{KRW} its  highest weight is 
\begin{equation}\label{w} ((M_1(k)+1)\theta_1,M_1(k)+1).\end{equation}
By \eqref{vinik},  $v_1$ lies in $I^k$; moreover, it  has weight \eqref{w}, hence it is a  highest weight vector. In particular it is singular and generates $I^k$.
When $|S|=2$, $\g$ is of type $D(2,1; \tfrac{m}{n})$ and $k=\tfrac{mn}{m+n}q,\,q\in\nat$. $H_0(J^k)$  is the sum of two highest weight modules of weights  
$(mq\theta_1,mq),\,(nq\theta_2,nq)$.  By Remark \ref{d21a}, we can assume $m>n$; then $v_2$  is singular, by the above comparing weight argument. We should  prove that 
$v_1$ is not in the submodule generated by $v_2$. Otherwise, we can reach $v_1$  applying to $v_2$ a combination of operators 
$L_m, m<0,\, J^{\{u\}}_{r}, u\in \g^\natural, r\leq 0, G^{\{v\}}_s,\,v\in \g_{-1/2}, s<0$.
 We can clearly assume that the $v$'s appearing are root vectors; let $\beta_v$ be the corresponding root, so that $\eta_v:=(\beta_v)_{|\h^\natural}$ is the corresponding weight. Let $\Pi=\{\a_1,\a_2,\a_3\}$ with $\a_1$ odd. Then $\eta_v\in\pm\{\tfrac{\a_2-\a_3}{2},\tfrac{\a_2+\a_3}{2}\}$. Define the weight of $ G^{\{v\}}_s$ to be the pair $(\eta_v,-s)$; likewise, if $u$ is a root vector of $\g^\natural$, and $\eta_u$ is the corresponding root, define the weight of 
  $J^{\{u\}}_{r}$ as $(\eta_u, -r)$. Finally, declare that the weight of $L_n$ is $(0,-n)$. Let $\D_W$ be the set of weights. Note that any
  element of $\D_W$ is a positive integral linear combination of $\Pi_W=\{(-\a_2,0),(-\a_3,0),(\tfrac{\a_2+\a_3}{2},\half)\}$. It follows that 
  $$(mq\a_2,mq) - (nq\a_3,nq) =a_1(-\a_2,0)+a_2(-\a_3,0)+a_3(\tfrac{\a_2+\a_3}{2},\half),\quad a_i\in \ZZ_+, $$ 
so that $a_1=-nq, a_2=mq, a_3=2(m-n)q$. Since $a_1$ is negative, we have the required contradiction.\qed

\subsection{$\g=spo(2|3)$}
Introduce the following vector
\begin{equation}
r_1=(x_{\a_1})_{(0)}u_1= (x_{\a_1})_{(0)}(x_{-\theta+\a_1})_{(1)}w_1.
\end{equation}
\begin{prop}\label{singprime2} The vector $r_1$ 
is singular in the universal affine vertex algebra $V^k(\g)$ and generates $J^k$.
\end{prop}
\begin{proof} As in the proof of Proposition \ref{singprime}, the vector $w_1$
 generates the maximal proper ideal of $V^k(\g)$.  Let $\mu$ be weight of $w_1$. An explicit calculation shows  that 
 $(\mu|\a_0+\a_1)=-k-\tfrac{1}{2}\ne 0$, hence $u_1\ne 0$,   and it is  singular for $r_{\a_1}\Pia$. So $r_1$ is either $0$ or singular for $\Pia$. The first possibility does not occur, since
 $$(k\L_0-(M_1(k)+1)(\d-\theta_1)|\a_1)=\tfrac{1}{2}(M_1(k)+1)\ne0.$$
The claim follows.
\end{proof}

  
\noindent {\it Proof of Theorem \ref{C} for  $\g= spo(2|3)$.}
Arguing as in the previous subsection, by Proposition \ref{singprime2} it follows that the maximal ideal in $W^k_{\min}(\g)$  generated by  a singular vector $v$ of weight  $(2 (m-2) \omega_1,m-3/2)$, where $\omega_1$ is the fundamental weight for $\mathfrak{sl}_2$ and $m=M_1(k)+2$. We observe that also the weight of $\widetilde v_1$ (cf. \eqref{tv1}) is $(2 (m-2) \omega_1,m-3/2)$. Moreover,    using \eqref{vinik} and the relations
\begin{align*}(G^{\{x_{-\a_1}\}})_{(0)}\widetilde v_1= c_1v_1,\quad   (G^{\{x_{-\xi}\}})_{(1)}  v_1 = c_2 \widetilde v_1, 
\text{ where $c_1,c_2\in\C\setminus \{0\},$}\end{align*}
we see that $\widetilde v_1 \in I^k$, hence it is a multiple of $v$, so  it is singular and generates $I^k$.
\qed

\section{Descending from $\Wu$ to $\Ws$}
 Let $H_0$ be the quantum Hamiltonian reduction functor from the category $\mathcal O_k$ of $\ga$-modules of level $k$ to the category of $W^k_{\min}(\g)$-modules \cite{KW1}.  By 
\cite{Araduke}, it is exact. 
As in \eqref{nuh}, for $\nu\in P^+_k$ and $h\in\C$ let $\widehat \nu_h=k\L_0+h\theta+\nu$. By \cite{Araduke}, \cite{KW1}, 
$H_0(L(\widehat \nu_h))=0$ if $\widehat \nu_h(\a_0^\vee)=k-2h\in\mathbb Z_{\ge 0}$, and 
$H_0(L(\widehat \nu_h)=L^W(\nu,\ell_0(h))$ if $\widehat \nu_h(\a_0^\vee)=k-2h\notin\mathbb Z_{\ge 0},$ where
\begin{equation}\label{ell0}
\ell_0(h)=\frac{(\widehat \nu_h|\widehat \nu_h+2\widehat \rho)}{2(k+h^\vee)}-h,\end{equation}
and $L^W(\nu,\ell_0)$ is  the irreducible highest weight $\Wu$-module with highest weight $(\nu,\ell_0)$
\cite{KMP1}, \cite{KW1}.
\par

 \begin{theorem}\label{t1} Let $k$ be in the unitary range. Then all irreducible highest weight $W^k_{\min}(\g)$--modules $L^W(\nu,\ell_0)$  with $\ell_0\in\C$ when $\nu\in P^+_k$ is not extremal, and $\ell_0=A(k,\nu)$ otherwise,
descend to $W_k^{\min}(\g)$. 
\end{theorem}
\begin{proof}  
 If $k-2h\in\ZZ_{\ge 0}$, then, by \cite[Lemma 11.8]{KMP1}, for $h'=k+1-h$ we have 
 $\ell_0:=\ell_0(h)=\ell_0(h')$. Since $k -2 h'\notin \ZZ_{+}$ we conclude that  $H_0(L(\widehat \nu_{h'}))=L^W(\nu,\ell_0)$.
 
 So
 \begin{equation}\label{dacit}\text{for each $\ell_0$ there is $\tilde h\in \C$ such that $L^W(\nu,\ell_0)=H_0(L(\widehat \nu_{\tilde h}))$}.\end{equation}
 By Theorem \ref{TA}, if $\nu$ is not extremal, then $L(\widehat \nu_{\tilde h})$ is a $V_k(\g)$--module, hence $L^W(\nu,\ell_0(\tilde h))$ is a $\Ws$-module. Note that  $h\in E_{k,\nu}$ if and only if $\ell_0(h)=A(k,\nu)$. It follows from Theorem \ref{TA} that,  if   $\nu$ is  extremal,  then $L(\widehat \nu_{\tilde h})$ is  a $V_k(\g)$--module, hence 
 $L^W(\nu,A(k,\nu))$ is a $\Ws$-module.
\end{proof}

A simple application of Theorem \ref{t1} is the proof of Conjecture 4 in \cite{KMP1}.

\begin{cor}\label{uu}
Any unitary $W^k_{\min}(\g)$--module $L^W(\nu,\ell_0)$ descends to 
$W_k^{\min}(\g)$.
\end{cor}
\begin{proof}
By Theorem 1.3 (1) of \cite{KMP1} (see Theorem \ref{nec} (1)), the unitary representations of $W^k_{\min}(\g)$ occur as representations listed in Theorem \ref{t1}.
\end{proof}
\begin{remark} Conjecture 2 from \cite{KMP1} that all $\Ws$-modules
$L^W(\nu, A(k, \nu))$ for extremal $\nu$
are unitary is still an open problem, except for $\g=spo(2|3)$
and $psl(2|2)$. We can prove this statement also when  $\g=spo(2|n)$ and $k=-1$.
\end{remark} 
\section{Classification of irreducible highest weight representations of $W_k^{\min}(\g)$ in the unitarity range.}\label{Section 6}
The main result of this section is the following theorem.
\begin{theorem}\label{t2} Let $k$ be in the unitary range. The modules appearing in Theorem \ref{t1} form the  complete list of inequivalent irreducible  highest weight representations of 
$W_k^{\min}(\g)$.
\end{theorem}

\begin{remark} Combining  \eqref{dacit} and Theorem \ref{t2}, we have proven that the irreducible highest weight modules of
$\Ws$ are precisely the non-zero images of the irreducible $V_k(\g)$ under  Hamiltonian reduction.
\end{remark}

We need to recall some well-known facts about Zhu algebras in the super case \cite{KWa}.\par 	\noindent
 Let $V=\bigoplus\limits_{n\in \tfrac{1}{2}\ZZ}V_n$ be a conformal vertex algebra, graded by the eigenspaces of $L_0$, with the parity 
 $p(V_n)\equiv 2n\mod 2$.  For $a\in V_n$ we write $\deg a =n$. Define bilinear maps $*:V\times V\to V,$ $\circ :V\times V	\to V$, setting
\begin{align*}
a*b&=\begin{cases}Res_z\left(Y(a,z)\frac{(z+1)^{\deg a}}{z}b\right)\quad&\text{if $a,b\in V_{\overline 0}$,} \\ 0 \quad &\text{if $a$ or $b\in V_{\overline 1}$.}
\end{cases}\\
a\circ b&=\begin{cases}Res_z\left(Y(a,z)\frac{(z+1)^{\deg a}}{z^2}b\right)\quad&\text{for $a\in V_{\overline 0}$,} \\ 
Res_z\left(Y(a,z)\frac{(z+1)^{\deg a-\tfrac{1}{2}}}{z^2}b\right)\quad&\text{for $a\in V_{\overline 1}$.} \\ 
\end{cases}
\end{align*}
Denote by $O(V)\subset V$ the $\C$-span of elements of the form $a\circ b$, and by $Zhu(V)$ the quotient space $V/O(V)$.       Then 
 $Zhu(V)$ is an associative algebra. 
Recall from  \cite[(3.1.12)]{FZ} or \cite[Lemma 1.1.(3)]{KWa} that if $[a]$ denotes the class of an element $a\in V$ in $Zhu(V)$, then for all $a,b\in V_{\overline 0}$ the following relations hold:
\begin{equation}\label{Z}
[a_{(-1)}\vac]*[b]=[(a_{(-1)}+a_{(0)})b],\quad  [b]*[a_{(-1)}\vac]=[a_{(-1)}b].
\end{equation}
\vskip10pt

Recall that a module $M$ over a conformal vertex algebra  is called a {\it positive energy} module if $L_0$ is diagonalizable on $M$ and all its eigenvalues lie in $h+\R_{\ge 0}$ for some $h\in\C$:
$$M=\bigoplus_{j\in h+\R_{\ge 0}}M_j,\quad M_h\ne \{0\}.$$
The subspace $M_h$ is called the {\it top component} of $M$.\par

 Recall  \cite{FZ}, \cite{KWa} that there is one-to-one correspondence between irreducible positive energy $V$--modules 
  and irreducible modules over  the Zhu algebra $Zhu(V)$, which associates to a $V$-module $M$ the $Zhu(V)$-module $M_h$.
Namely, to $Y^M(a,z)=\sum_j a^M_jz^{-j-\deg a}$ one associates ${a_0^M}_{|M_h}$. By the above construction of the Zhu algebra, 
it follows from \cite[Theorem 7.1]{KW1} that 
\begin{equation}\label{star}
Zhu(\Wu)\simeq \C[L]\otimes U(\g^\natural).
\end{equation}

Under the correspondence between irreducible positive energy $\Wu$-modules and irreducible modules over its Zhu algebra, the module $L^W(\nu,h)$ goes to the irreducible highest weight $\g^\natural$-module with highest weight $\nu$ on which $L$ acts by the scalar $h$, which we denote by $V(\nu,h)$. It follows from \eqref{star} that 
\begin{equation}\label{sstar}
Zhu(\Ws)\simeq \left(\C[L]\otimes U(\g^\natural)\right)/J(\g),\end{equation}
where $J(\g)$ is a 2-sided ideal of the associative algebra $\C[L]\otimes U(\g^\natural)$. So any non-zero element in $J(\g)$ imposes a condition on the highest weight $(\nu,h)$ of the $Zhu(\Ws)$-module $V(\nu,h)$.

 \vskip10pt
\noindent{\it Proof of Theorem \ref{t2}: the case  $\g= spo(2\vert 3)$.}
First we present a proof in the case $\g = spo(2 \vert 3)$,  which gives a motivation for  the proof  in the general case.

Let $k=-\frac{m}{4},$ where $m\in\mathbb Z_{\ge 3}$, so that $M_1(k)=m-2$ and $P^+_k=\{j\omega_1\mid 0\leq j\leq m-2\}$ (here $\omega_1$ is the fundamental weight for $\g^\natural\cong sl_2$).   Set 
 $ \mathcal W^k =  W_{\min}^{k}(spo(2|3)),$\newline $\mathcal W_k =  W^{\min}_{k}(spo(2|3))$.
Denote by $L_k [j, q]$ the irreducible highest weight  $\mathcal W_k$-module of level $k$ generated by a highest weight vector $v_{j,q}$, such that for $n \in {\Z}_{\ge 0}$:
\bea
 && L_n v_{j,q} = q \delta_{n,0}v_{j,q}, \ G^{+}_{(n+1/2)}v_{j,q} = G^{-}_{(n+1/2)}v_{j,q} = G^{0}_{(n+1/2)}v_{j,q} =0 \nonumber \\
 &&  J^0_{(n)}v_{j,q} = j \delta_{n,0} v_{j,q}, \ J^+_{(n)}v_{j,q} = J^-_{(n+1)}v_{j,q}=0. \nonumber \eea
Here we use notation for the generators of $\mathcal W_k$ as in \cite[Section 8.5]{KW1}.
Note that $U(sl(2)) v_{j,q}=V  (j \omega_1)$ is the irreducible highest weight  $sl(2)$--module with highest weight $j \omega_1$.


\begin{lemma} \label{rel-m4}  Set $\Omega = -\tfrac{m-2}{ 4} (   [L] +\frac{[J^0]}{4} )   + \tfrac{1}{8} [J^+]*[J^-]$. 
 Then 
 \begin{equation}\label{relZ} \Omega  * [J^-] ^{m-3} \in J(\g).\end{equation}
\end{lemma}
\begin{proof} By Theorem \ref{C}, we have that $(J^-)_{(-1)}^{m-3} G^- \in I^k$, hence 
$[G^+ _{(0)}(J^-)_{(-1)}^{m-3} G^-] =0$. Since $G^+ _0$ acts as a derivation, using \eqref{Z} we get that
\begin{align} &0=[(G^+)_{(0)}(J^-)_{(-1)}^{m-3} G^-] = [(G^0)_{(-1)} (J^-)_{(-1)}^{m-4} G^-]  + \cdots +\nonumber\\     &[(J^-)_{(-1)}^{m-5} (G^{0}) _{(-1)}(J^-)_{(-1)} G^-]   +
[(J^-)_{(-1)}^{m-4} (G^+)_{(0)} G^-]  \nonumber \\
&=  (m-3)  [ (J^-)_{(-1)}^{m-4}  (G^0)_{(-1)} G^-] + \tfrac{ (m-3) (m-4)}{2}  [   (J^-)_{(-1)}^{m-5}  (G^-) _{(-2)} G^-] \nonumber\\&+ [ (J^-)_{(-1)}^{m-4}(G^+)_{(0)} G^-] 
\nonumber \\
&= (m-3) [(G^0)_{(-1)} G^+] * [J^{-}] ^{m-4} + \tfrac{ (m-3) (m-4)}{2}  [  (G^-)_{(-2)} G^-] * [J^{-}] ^{m-5}\nonumber\\ 
&+ [(G^+)_ {(0)} G^-] *[J^{-}] ^{m-3}.  \label{rel} \end{align}
Using $\lambda$--bracket formulas from \cite[Section 8.5]{KW1} we get that the following relations hold in $\mathcal W^k$:
\begin{align*} (G^+)_{(0)} G^- &= \tfrac{2k+1}{2} L + \tfrac{1}{8} :J^-  J^+: +\tfrac{k+1}{4} \partial J^0,  \\
  (G^0)_{(0)} G^- &=    \tfrac{1}{8} :J^- J^0:  -\tfrac{k+1}{2} \partial J^-, \\
  (G^-) _{(0)}G^- &=   - \tfrac{1}{4} :J^-   J^-:.
 \end{align*}
   This implies, in $Zhu(\mathcal W^k)$, 
\begin{align} \label{f1} [G^+ _{(0)}G^-] &=  \tfrac{2k+1}{2} ( [L] +\tfrac{1}{4 (2 k+1)} [J^+]* [J^-]  -\tfrac{k+1}{2 (2k+1)} [J^0] ), \\ 
 [G^0 _{(-1)} G^-] &= - [(G^0)_{(0)} G^ -] =     -\tfrac{1}{8} [J^0] *[J^-] - \tfrac{k+1}{2} [J^-]=   - \tfrac{1}{8} ( [J^0]  +4(k+1) ) *[J^- ],  \label{f2}\\  \ 
[G^- _{(-2)} G^- ] &= - [(G^-)_{(-1)} G^- ]  = [(G^-)_{(0)} G^- ] =   - \tfrac{1}{4} [J^-]  * [J^-]. \label{f3}
 \end{align}
Substituting  \eqref{f1}, \eqref{f2}, \eqref{f3} into \eqref{rel} and collecting $[J^-]^{m-3}$, we get \eqref{relZ}, where
\bea  \Omega &=&  \tfrac{2k+1}{ 2} [L] + \tfrac{1}{8} [J^+]*[J^-] - \tfrac{k+1}{4} [J^0]     - \tfrac{ (m-3) (m-4)}{8} - \tfrac{m-3}{8} ([J^0] + 4 (k+1) ).
 \nonumber  \\
 &=& -  \tfrac{m-2}{ 4}  [L]+   \tfrac{1}{8} [J^+]*[J^-]  + \tfrac{m-4}{16} [J^0]  - \tfrac{ (m-3) (m-4)}{8} - \tfrac{m-3}{8} ([J^0] - m+4) \nonumber 
\\\nonumber
&=&-\tfrac{m-2}{ 4}  [L] -\tfrac{m-2}{16} [J^0]   + \tfrac{1}{8} [J^+]*[J^-].
 \eea
  \end{proof}
\begin{proposition}  
\label{zhu-osp}

Let $  L_{k}[j,q]$ be an irreducible highest weight $\mathcal W_{k}$--module. Then it is isomorphic to exactly one of the following modules:

\begin{itemize}
\item  $L_{k}[j, q]$ with  $0 \le j \le m-4$ and  $q \in {\C}$;
\item $L_{k} [m-3, \frac{m-3}{4}]$;
\item   $L_{k} [m-2, \frac{m-2}{4}]$.
\end{itemize}
\end{proposition}
 \begin{proof} First,  $j \in \{0, \dots, m-2\}$ and the top component is 
  $ L_k[j, q]_{top} =V(j \omega_1) \otimes {\C}_q, $
  where  $V(j \omega_1)$ is the $(j+1)$--dimensional irreducible $sl(2)$--module with highest weight $j \omega_1$, and $\C_q$ the $1$-dimensional ${\C}[L]$--module on which $L_0$--acts as multiplication with $q \in {\C}$. If $0 \le j \le m-4$, then
 $[J^-] ^{m-3} $ acts trivially on $ L_k[j, q]_{top} $ for  each $q \in {\C}$, hence  the same holds for the action 
of $ [{G^+}_{(0)} G^-]* [J^-] ^{m-3} $.
If $ m-3\le j \le m-2$, then   $(J^-)_{(0)}^{m-3}$  acts non-trivially on $ L_k[j, q]_{top} $. Hence there exists  $w \in  L_k[j, q]_{top}$ such that $ w' =(J^-)_{(0)}^{m-3} w$ is a lowest weight vector for $sl(2)$, i.e.
  $$(J^0)_{(0)}w' = -j w',\ (J^-)_{(0)} w' = 0. $$ 
 Then we have, by Lemma \ref{rel-m4}:
 \bea 
 0&=&  ( \Omega * [J^-] ^ {m-3} )w \nonumber \\
 &=&  \left(  -\tfrac{m-2}{4} ( L_0 + \tfrac{1}{4}(J^0)_{(0)} ) +\tfrac{1}{8}  (J^+)_{(0)} (J^-)_{(0)}  \right)  (J^-)_{(0)}^{m-3} w \nonumber \\
 &=&  \left(  -\tfrac{m-2}{4} ( L_0 + \tfrac{ 1}{4} (J^0)_{(0)}) +\tfrac{1}{8}  (J^+)_{(0)} (J^-)_{(0)} \right)  w' \nonumber \\
 &=&  -\tfrac{m-2}{4} (q -\tfrac{j}{4})w', \nonumber 
 \eea
hence, when $j=m-3$ or $m-2$, we have that $q= \tfrac{j}{4}$.
\end{proof}


Since the modules  appearing in Theorem \ref{t2} in case of $\g=spo(2|3)$ are exactly  those listed  in Proposition \ref{zhu-osp}, 
Theorem \ref{t2} is proved in this case.

\bigskip

\noindent{\it Proof of Theorem \ref{t2}: the case   $\g \ne  spo(2 \vert 3)$.}
We first illustrate  the strategy of the proof in the case $\g=psl(2|2)$. We use notation of \cite[Section 8.4]{KW1}.

  Set $ \mathcal W_{\min}^{k} =  W^{k}_{\min}(psl(2\vert 2))$ and  $\mathcal W^{\min}_{k} =  W^{\min}_{k}(psl(2\vert 2))$. 
Let $m = -k-1=M_1(k) \in {\Z}_{>0}$. Then  $V_k(sl(2))$ is a vertex subalgebra  of $\mathcal W^{\min}_{k}$ generated by $J^{\pm}, J^0$. The odd generators of conformal weight $3/2$ are $  G^{\pm}, \overline G^{\pm}$.
By Theorem \ref{C} the maximal ideal $I^k$ is generated by the singular vector $(J^+)_{(-1)}^{m+1} {\bf 1}$. Then  $(G^-)_{(0)} (\overline G^{-})_{(0)}(J^+_{(-1)})^{m+1} {\bf 1} \in I^k$. By the $\l$-bracket formulas given in \cite[Section 8.4]{KW1} we have
\begin{align}(G^-) _{(0)}  \overline G^{-}_{(0)} (J^+_{(-1)})^{m+1} {\bf 1} &= (m+1)  (G^-) _{(0)}   (J^+_{(-1)}) ^{m}   \overline G^+ \nonumber   \\ &=- (m+1) m   (J^+_{(-1)}) ^{m-1}
  (G^+)_{(-1)} \overline G^+ +  (m+1) (J^+_{(-1)}) ^{m}  (G^-) _{(0)}  \overline G^+. \label{rr3} \end{align}

 Using the  definition of the Zhu algebra $Zhu(\mathcal W^k_{\min})$, we have that for $G =G^{\pm}$ or $G =\overline G^{\pm}$  and each $v \in \mathcal W^k_{\min}$:
\begin{equation}\label{69a} \mbox{Res}_z \frac{(z+1) ^{\deg G-1/2}}{z} G(z) v = (G_{(-1)} + G_{(0)} ) v = 0 \ \text{in } Zhu(\mathcal W^k_{\min}). \end{equation}
Using \eqref{69a} and the $\l$-bracket formulas given in \cite[Section 8.4]{KW1}, we get
 \begin{align}
 &[  (G^+)_{(-1)} \overline G^+] = - [(G^+)_{(0)} \overline G^+] =- [(J^+)_{(-2)}{\bf 1}] =  [J^+],  \label{rr1}\\
 &[  (G^-) _{(0)} \overline G^+] =   [L] - \tfrac{1}{2} [(J^{0})_{(-2)}{\bf 1}] = [L] +  \tfrac{1}{2} [J^0],\label{rr2}\end{align}
Using \eqref{rr3}, \eqref{rr1},\eqref{rr2} and \eqref{Z}, we find
 \begin{align*}
 [(G^-)_{(0)}   (\overline G^{-})_{(0)} (J^+_{(-1)})^{m+1} {\bf 1}] 
&= - m (m+1) [J^+ ]^m + (m+1)   ([L]  +  \tfrac{1}{2} [J^0] ) *   [J^+]^{m} \\
&= (m+1) \left( [L]  +  \tfrac{1}{2} [J^0]  -m \right)  *   [J^+]^{m}
 \end{align*}
 so that, setting $ \Omega =  [L]  +  \tfrac{1}{2} [J^0]  -m$, we have from \eqref{69a}  \begin{equation}\label{relpsl22}
 \Omega*[J^+]^m\in J(\g).\end{equation}
 \begin{proposition}  \label{zhu-N4}
Let $  L^{N4} _{k}[j,\ell_0]$ be an irreducible highest weight $\mathcal W^{\min}_{k}$--module. Then it is isomorphic to exactly one of the following modules:
\begin{itemize}
\item  $L^{N4}_{k}[j, \ell_0]$ with  $0 \le j \le m-1$ and  $\ell_0 \in {\C}$;
\item  $L^{N4}_{k} [m, \frac{m}{2}]$.
\end{itemize}
\end{proposition}
 \begin{proof} Let $v[j,\ell_0]$ be a  highest weight vector for $L^{N4} _{k}[j,\ell_0]$. Then $j \in \{0, \dots, m\}$ and the top component is 
  $ L^{N4}_k[j, \ell_0]_{top} =V(j \omega_1) \otimes {\C}_{\ell_0}, $
  where  $V(j \omega_1)$ is the $(j+1)$--dimensional irreducible $sl(2)$--module with highest weight $j \omega_1$, and ${\C}_{\ell_0}$ the $1$-dimensional ${\C}[L]$--module on which $L(0)$--acts as multiplication by $\ell_0 \in {\C}$. If $0 \le j \le m-1$, then
 $  \Omega * [J^+] ^{m} $ acts trivially on $ L^{N4}_k[j, \ell_0]_{top} $ for  each $\ell_0 \in {\C}$.
 
If $ j=m$, then   $(J^+_{(0)})^{m}$  acts non-trivially on $ L^{N4}_k[j, \ell_0]_{top} $. Then  there  exists  $w \in  L^{N4} _k[j, \ell_0]_{top}$ such that $ (J^+_{(0)})^m  w = v[j,\ell_0]$ is a  highest  weight vector for $sl(2)$.   We get by \eqref{69a}
 \bea 
 0&=&  ( \Omega * [J^+] ^ {m} ) w \nonumber \\
 &=&  (L_{(0)}+   \tfrac{1}{2}   (J^0_{(0)})= -m  (J^+_{(0)})^{m} w \nonumber \\
 &=&  \left(  \ell_0-m/2 \right)  v[j,\ell_0]. \nonumber 
  \eea
This implies that for $ j=m $ we need to have $\ell_0= \tfrac{m}{2}$.
\end{proof}

Proposition \ref{zhu-N4} proves, in particular, Theorem \ref{t2} for $\g=psl(2|2)$.
We now deal with  the general case $\g\ne spo(2|3)$. We shall see that a relation similar to \eqref{relpsl22} holds in $Zhu(\Wu)$. We do not need a very precise expression for $\Omega$: what is really relevant is the fact that the action of $\Omega$  gives  a relation which is linear in $\ell_0$.\par
We
start by  observing that there exist two odd positive roots $\gamma_1,\gamma_2$  such that 
\begin{equation}\label{rcfat}\theta-\gamma_1-\gamma_2=-\theta_i,\ i\in S.\end{equation} This fact can be verified directly by looking at Table 3. 
\begin{table}\label{rootdata}
\begin{tabular}{c | c| c |c | c | c |c}
$\g$&
$spo(2|3)$&
$psl(2|2)$&
$spo(2|m), m>4$&
$D(2,1;\tfrac{m}{n})$&
$F(4)$&
$G(3)$\\\hline
$\gamma_1$&$\d_1+\e_1$&$\e_1-\d_2$&$\d_1+\e_1$&$\e_1+\e_2-\e_3, \e_1+\e_2+\e_3$
&$\frac{1}{2}(\d_1+\e_1+\e_2-\e_3)$&$\d_1-\e_3$
\\\hline
$\gamma_2$&$\d_1$&$\d_1-\e_2$&$\d_1+\e_2$&$\e_1+\e_2+\e_3, \e_1-\e_2+\e_3$ &$\frac{1}{2}(\d_1+\e_1+\e_2+\e_3)$ & $\d_1+\e_2$
\\\hline
$\theta_i$&$\e_1$&$\d_1-\d_2$&$\e_1+\e_2$&$2\e_2,2\e_3$&$\e_1+\e_2$&$\e_2-\e_3$
\\\hline
$\theta$&$2\d_1$& $\e_1-\e_2$&$2\d_1$&$2\e_1$&$\d_1$&$2\d_1$
\end{tabular}
\caption{Root data}
\end{table}
Using \eqref{rcfat} and  the explicit expression for  $[{G^{\{u\}}}_{\lambda}G^{\{v\}}]$ given in \cite[(1.1)]{AKMPP} we find
\begin{align}\label{11}
&{G^{\{x_{\theta_i-\gamma_1}\}}}_{(0)}G^{\{x_{\theta_i-\gamma_2}\}}=
\sum_{s=1}^{\dim\g_{1/2}}:J^{\{[x_{\theta_i-\gamma_1},u^s]^\natural\}}J^{\{[u_s,x_{\theta_i-\gamma_2}]^\natural\}}:
+2c_1(k+1)\partial J^{\{x_{\theta_i}\}},\\
&\label{22}{G^{\{x_{-\gamma_1}\}}}_{(0)}G^{\{x_{\theta_i-\gamma_2}\}}=-2(k+h^\vee)c_2L+c_2\sum_{\alpha=1}^{\dim \g^\natural} 
:J^{\{a^\alpha\}}J^{\{a_\alpha\}}:+\\\notag
&\sum_{s=1}^{\dim\g_{1/2}}:J^{\{[x_{-\gamma_1},u^s]^\natural\}}J^{\{[u_s,x_{\theta_i-\gamma_2}]^\natural\}}:
+2(k+1)c_3\partial J^{\{h^\natural_{-\theta_i+\gamma_2}\}},
\end{align}
where $c_1,c_2,c_3$ are constants independent of $k$. Here $a^\natural$ denotes the orthogonal projection of $a\in \g$ to $\g^\natural$ and $\{u_s\},\, \{u^s\}$ are basis of $\g_{1/2}$ dual with respect to  the bilinear form \begin{equation}\label{pizz}\langle u,v\rangle=(x_{-\theta}|[u,v]).\end{equation}
Set $x_{\theta_i-\gamma_j}=[x_{-\gamma_j},x_{\theta_i}],\,j=1,2$.  Recall that , by  \cite[Theorem 2.1 (e)]{KW1},  if  $u\in \g_{-1/2}$ and $a\in\g^\natural$, then 
\begin{equation}\label{JG}
[J^{\{a\}}_{(-1)}, G^{\{u\}}_{(0)}]=G^{\{[a,u]\}}_{(-1)}.\end{equation}
\begin{lemma} If $\g\ne spo(2|3)$, $N\in\mathbb Z_{\ge 1}$ and  $M\in \mathbb Z_{\geq 2}$, then, for $i\in S$ we have 
\begin{equation} \label{GJN}[G_{(0)}^{\{x_{-\gamma_j}\}},(J_{(-1)}^{\{x_{\theta_i}\}})^N]=-N(J_{(-1)}^{\{x_{\theta_i}\}})^{N-1}G_{(-1)}^{\{x_{\theta_i-\gamma_j}\}},\quad j=1,2.
\end{equation}
\begin{align}
\notag&(G^{\{x_{-\gamma_1}\}})_{(0)} (G^{\{x_{-\gamma_2}\}})_{(0)}(J^{\{x_{\theta_i}\}}_{(-1)})^{M}\vac=\\
 &M\left((M-1)(J^{\{x_{\theta_i}\}}_{(-1)})^{M-2}(G^{\{x_{\theta_i-\gamma_1}\}})_{(-1)}
 (G^{\{x_{\theta_i-\gamma_2}\}})_{(-1)}\vac \right. \label{33}\\
&-\left. (J^{\{x_{\theta_i}\}}_{(-1)})^{M-1}(G^{\{x_{-\gamma_1}\}})_{(0)}(G^{\{x_{\theta_i-\gamma_2}\}})\right).\notag
\end{align}
\end{lemma}
\begin{proof} We prove \eqref{GJN}  by induction on $N$, with base $N=1$ given by  \eqref{JG}. If $N>1$ we have 
\begin{align*}&[G_{(0)}^{\{x_{-\gamma_j}\}},(J_{(-1)}^{\{x_{\theta_i}\}})^N]=[G_{(0)}^{\{x_{-\gamma_j}\}},(J_{(-1)}^{\{x_{\theta_i}\}})^{N-1}](J_{(-1)}^{\{x_{\theta_i}\}})+
(J_{(-1)}^{\{x_{\theta_i}\}})^{N-1}[G_{(0)}^{\{x_{-\gamma_j}\}},(J_{(-1)}^{\{x_{\theta_i}\}})]\\
&=(1-N)(J_{(-1)}^{\{x_{\theta_i}\}})^{N-2}G_{(-1)}^{\{x_{\theta_i-\gamma_j}\}}(J_{(-1)}^{\{x_{\theta_i}\}})-(J_{(-1)}^{\{x_{\theta_i}\}})^{N-1}G_{(-1)}^{\{x_{\theta_i-\gamma_j}\}}\\
&=-N(J_{(-1)}^{\{x_{\theta_i}\}})^{N-1}G_{(-1)}^{\{x_{\theta_i-\gamma_j}\}}.\end{align*}
In the final equality we have used that $2\theta_i-\gamma_j\notin \D$  for $\g\ne spo(2|3)$.\par
We now prove \eqref{33}: \begin{align*}
(G^{\{x_{-\gamma_1}\}})_{(0)} (G^{\{x_{-\gamma_2}\}})_{(0)}(J^{\{x_{\theta_i}\}}_{(-1)})^{M}\vac&=
-M(G^{\{x_{-\gamma_1}\}})_{(0)} (J^{\{x_{\theta_i}\}}_{(-1)})^{M-1}(G^{\{x_{\theta_i-\gamma_2}\}})_{(-1)}\vac\\
&=M(M-1) (J^{\{x_{\theta_i}\}}_{(-1)})^{M-2}(G^{\{x_{\theta_i-\gamma_1}\}})_{(-1)}(G^{\{x_{\theta_i-\gamma_2}\}})_{(-1)}\vac\\
&-M(J^{\{x_{\theta_i}\}}_{(-1)})^{M-1}(G^{\{x_{-\gamma_1}\}})_{(0)}(G^{\{x_{\theta_i-\gamma_2}\}})\vac.
\end{align*}
\end{proof}
By Theorem \ref{C}, if $M=M_i(k)+1$, then the element displayed in \eqref{33} lies in $I^k$, hence its projection to $Zhu(\Ws)$ lies in $J(\g)$.  Next  we calculate explicitly this projection. First remark  that, by \eqref{69a},
\begin{equation}\label{4}[(G^{\{x_{\theta_i-\gamma_1}\}})_{(-1)}
 (G^{\{x_{\theta_i-\gamma_2}\}})_{(-1)}\vac]=-[(G^{\{x_{\theta_i-\gamma_1}\}})_{(0)}
 (G^{\{x_{\theta_i-\gamma_2}\}})].
\end{equation}
Substituting \eqref{4} into \eqref{33} and using \eqref{11}, \eqref{22}, we get 
\begin{align}\label{prima}
&\frac{1}{M_i(k)+1}[(G^{\{-\gamma_1\}})_{(0)} (G^{\{-\gamma_2\}})_{(0)}(J^{\{x_{\theta_i}\}}_{(-1)})^{M_i(k)+1}\vac]=\\\notag
 &-M_i(k)\left(\sum_{\gamma=1}^{\dim\g_{1/2}}[J^{\{[u_s,x_{\theta_i-\gamma_2}]^\natural\}}]*[J^{\{[x_{\theta_i-\gamma_1},u^s]^\natural\}}]*[J^{\{x_{\theta_i}\}}]^{M_i(k)-1}
-2(k+1)c_1[J^{\{x_{\theta_i}\}}]^{M_i(k)}\right)\\\notag
 &+c_2\left(-(k+h^\vee)[L]+\tfrac{1}{2}\sum_{\alpha=1}^{\dim \g^\natural} 
[J^{\{a_\alpha\}}][J^{\{a^\alpha\}}]\right)*[J^{\{x_{\theta_i}\}}]^{M_i(k)}+\\\notag
&-\left(\sum_{s=1}^{\dim\g_{1/2}}[J^{\{[u_s,x_{\theta_i-\gamma_2}]^\natural\}}]*[J^{\{[x_{-\gamma_1},u^s]^\natural\}}]
+2(k+1)c_3[\partial J^{\{h^\natural_{-\theta_i+\gamma_2}\}}]\right)*[J^{\{x_{\theta_i}\}}]^{M_i(k)}.
 \end{align}
 To finish the calculation we need the following fact.
\begin{lemma} For suitable constants $d_1,d_2,d_3$, independent of $k$,  we have
\begin{align}\notag\sum_{s=1}^{\dim\g_{1/2}}[J^{\{[u_s,x_{\theta_i-\gamma_2}]^\natural\}}]*[J^{\{[x_{\theta_i-\gamma_1},u^s]^\natural\}}]&=d_1\, [J^{\{h^\natural_{-\theta_i+\gamma_2}\}}]*[J^{\{x_{\theta_i}\}}]+
 d_2\, [J^{\{x_{\theta_i}\}}]*[J^{\{h^\natural_{-\theta_i+\gamma_1}\}}]\\
&\label{55}= (d_1\, [J^{\{h^\natural_{-\theta_i+\gamma_2}\}}]+
 d_2\,[J^{\{h^\natural_{-\theta_i+\gamma_1}\}}]+d_3)*[J^{\{x_{\theta_i}\}}].\end{align}
\end{lemma}
\begin{proof} We can assume that $u_s, u^s$ are root vectors. A direct inspection shows that $[u_s,x_{\theta_i-\gamma_2}]^\natural,$ $ [x_{\theta_i-\gamma_1},u^s]^\natural$ are both non-zero only if $u_s\in\g_\a,\,u^s\in\g_\beta$
 with either  $-\a=\theta_i-\gamma_{2}$ and $\beta=\gamma_1$ or  $\a=\gamma_{2}$ and 
$-\beta=\theta_i-\gamma_{1} $. \end{proof}
Set
\begin{align}\label{omega}
\Omega_i&=(-M_i(k)(d_1\, [J^{\{h^\natural_{-\theta_i+\gamma_2}\}}]+
 d_2\,[J^{\{h^\natural_{-\theta_i+\gamma_1}\}}]+d_3)
 -c_2(k+h^\vee)[L]\\&+\tfrac{1}{2}c_2\sum_{\alpha=1}^{\dim \g^\natural} 
[J^{\{a_\alpha\}}][J^{\{a^\alpha\}}]+\sum_{s=1}^{\dim\g_{1/2}}[J^{\{[u_s,x_{\theta_i-\gamma_2}]^\natural\}}]*[J^{\{[x_{-\gamma_1},u^s]^\natural\}}]
-2(k+1)c_3[J^{\{h^\natural_{-\theta_i+\gamma_2}\}}]).\notag
\end{align}
Note that all the constants involved in \eqref{omega} do not depend on $k$; moreover, by \eqref{rcfat}, 
\begin{equation}\label{c2}c_2=(x_\theta|[x_{\theta_i-\gamma_1},x_{\theta_i-\gamma_2}])\ne 0.\end{equation}
Substituting \eqref{55} in \eqref{prima} we obtain
\begin{prop}  \label{Zhu-general} In $Zhu(\Wu)$,  we have
\begin{equation}
\Omega_i*[J^{\{x_{\theta_i}\}}]^{M_i(k)}\in J(\g),\ i\in S.
 \end{equation}

\end{prop}
\vskip20pt
\begin{proof}[Proof of Theorem \ref{t2}]
We have to  prove that if $L^W(\nu,\ell_0)$ is a irreducible highest weight  $W_k^{\min}(\g)$-module, then the pair $(\nu,\ell_0)$ is among those listed in the statement. Note that there is a non-zero vertex algebra homomorphism
$\Theta:V^{M_i(k)}(\g^\natural_i)\to W_k^{\min}(\g)$. Since $W_k^{\min}(\g)$ is unitary, we have that $\Theta(V^{M_i(k)}(\g^\natural_i))=V_{M_i(k)}(\g^\natural)$. In particular,  $L^W(\nu,\ell_0)$ is integrable as a $V_{M_i(k)}(\g^\natural_i)$-module, hence $\nu\in P^+_k$. If $\nu$ is not extremal, we are done. Assume therefore that $\nu$ is extremal.
The action of the Zhu algebra on the top component $V(\nu,\ell_0)$ of  $L^W(\nu,\ell_0)$ is given by the action of  the elements $[J^{\{a\}}]$ on  the irreducible finite-dimensional 
$\g^\natural$-module $V(\nu, \ell_0)$ of highest weight $\nu$, while $[L]$ acts as the multiplication by $\ell_0$. Consider the $sl(2)$-triple $\{x_{\theta_i},h_{\theta_i},x_{-\theta_i}\}$.
Since $\nu$ is extremal, $\nu(\theta_i^\vee)=M_i(k)$ for some $i\in S$. By $sl(2)$-theory applied to $U(sl(2))v$  where $v\in V(\nu,\ell_0)$ is  a  highest weight vector, 
we see  that $v$ is a non-zero multiple of $x_{\theta_i}^{M_i(k)}v'$, with $v'=x_{-\theta_i}^{M_i(k)}v$.  It follows that 
\begin{equation}\label{ef}0=\Omega_i [J^{\{x_{\theta_i}\}}]^{M_i(k)}v'=\Gamma_i(\ell_0) [J^{\{x_{\theta_i}\}}]^{M_i(k)}v',\end{equation}where, by \eqref{omega}, $\Gamma_i(\ell_0)=-c_2(k+h^\vee)\ell_0+b$ and  $b\in\C$ is independent of $\ell_0$, and
$c_2\ne 0$ by \eqref{c2}.  Since, by construction, 
$[J^{\{x_{\theta_i}\}}]^{M_i(k)}v'\ne 0$, we find from \eqref{ef} that  $\Gamma_i(\ell_0)=0$. Since $L^W(\nu,A(k,\nu))$ is a representation of $\Ws$, then  $\Gamma_i(A(k,\nu))=0$.  But  $\Gamma_i$ is a linear relation, so
$\ell_0=A(k,\nu)$ is the unique solution of the equation $\Gamma_i(\ell_0)=0$.
\end{proof}


A further refinement of Theorem \ref{t2} is the following result.
\begin{theorem}\label{ipehw} Let $k$ be in the unitary range.
The irreducible highest weight $\Ws$-modules are all the  irreducible positive energy representations.
\end{theorem}
\begin{proof}
Let $\mathcal V_k(\g^{\natural}) = \bigotimes_{i\in S} V_{M_i(k)}(\g^{\natural}_i)$.    Since each  $V_{M_i(k)}(\g^{\natural}_i)$ is a rational vertex algebra, then $ Zhu(\mathcal V_k(\g^{\natural})) = \bigotimes_i  Zhu(V_{M_i(k)}(\g^{\natural}_i) )$ is a semi-simple finite-dimensional associative algebra.

Let  $M$ be  an irreducible  positive energy  $W_k^{\min}(\g)$--module. Then the top component $M_{top}$ is an irreducible module for Zhu's algebra $Zhu(W_k^{\min}(\g))$.  By using the  embedding  $\mathcal V_k(\g^{\natural})  \rightarrow \Ws$ we see that $M_{top}$ is also a $Zhu(\mathcal V_k(\g^{\natural}) )$--module. 
    Since   $[L]$
 is a central  element in $ Zhu(W_k^{\min}(\g))$, it acts on $M_{top}$ as multiplication by a scalar, hence, by  (\ref{sstar}), $M_{top}$ remains irreducible
 when restricted to $\g^\natural$  and  therefore  it is an irreducible $Zhu(\mathcal V_k(\g^{\natural}) )$--module.  Since $Zhu(\mathcal V_k(\g^{\natural}) )$ is a semi-simple finite-dimensional associative algebra, we conclude that  $M_{top}$ is  finite-dimensional.
This implies that $M_{top}$ contains a  highest weight $\mathcal V_k(\g^{\natural}) $--vector $w$, which is then  a highest weight  vector for the action of $W_k^{\min}(\g)$. Therefore
$W_k^{\min}(\g)w$ is a highest weight submodule of $M$. Irreducibility of $M$ implies that $M =W_k^{\min}(\g)w$ is a highest weight $W_k^{\min}(\g)$--module. 
\end{proof}
 
\section{On positive energy modules over $V_k(\g)$ and $\Ws$}
By Theorem \ref{ipehw}, the irreducible highest weight $\Ws$-modules are all the  irreducible positive energy representations. In this section we show, by contrast, that $V_k(\g)$ admits positive energy modules which are not highest weight,  and we show that $\Ws$ admits non-positive energy modules.

 The $V_k(\g)$--modules which we construct belong to class of modules called  {\it relaxed} highest weight modules. These modules appear in \cite{AdM95, FST, LMRS, R1}, in the context of the representation theory of $V_k(\mathfrak{sl}_2)$;  later they have been systematically studied in \cite{KaR1, KaR2, KaRW} for higher rank cases. We will show the existence of such modules for $V_k(\g)$ for some values of  $k$ belonging to  the unitary range. We explore a free field realization which enables us to construct relaxed highest weight modules from relaxed highest weight modules over the Weyl vertex algebra (also called $\beta \gamma$ ghost vertex algebra). These relaxed modules were  previously studied in \cite{RW, AW, AdP-2019}. 

On the contrary, the vertex algebra $\Ws$ does not have non-highest weight positive energy modules. The non-highest weight  relaxed highest weight  $V_k(\g)$--modules are  either  mapped to zero by  quantum Hamiltonian reduction, so they don't contribute to $\Ws$--modules; or they are mapped to highest weight $\Ws$--modules. This connection deserves a more detailed analysis in the future.

We shall see below that $\Ws$ contains non-positive energy modules having all infinite dimensional weight spaces.

\subsection{Positive energy modules over $V_k(\g)$ which are not highest weight}
 We follow the notation and results of \cite[Section 6]{AMPP-20}.
   Consider the superspace $\C^{m|2n}$ equipped with the standard supersym\-metric form $\langle\cdot,\cdot\rangle_{m|2n}$ given in \cite{Kacsuper}. Let $V=\Pi \C^{m|2n}$, where $\Pi$ is the parity reversing functor. Let $M_{(m|2n)}$ be the universal vertex algebra generated by $V$ with $\l$--bracket
\begin{equation}\label{lambdaproduct}
[v_\lambda w]=\langle w,v\rangle.
\end{equation}
Let $\{e_i\}$ be the standard basis of $V$ and let $\{e^i\}$ be its dual basis with respect to $\langle \cdot,\cdot\rangle$ (i. e. $\langle e_i,e^j\rangle=\d_{ij}$). In this basis the $\lambda$-brackets are given by 
$$
[{e_h}_\lambda e_{m-k+1}]=\delta_{hk},\quad
[{e_{m+i}}_\lambda e_{m+2n-j+1}]=-\delta_{ij},\quad
[{e_{m+n+i}}_\lambda e_{m+n-j+1}]=\delta_{ij},
$$
for $h,k=1,\ldots m,\,\,i,j= 1\ldots,n$.

The vertex algebra $M_{(m \vert n)} $ is called  the Weyl-Clifford vertex algebra  and,  in the physics terminology,  the
$\beta\gamma b c$ system (cf. \cite{ciao2}).  In the case $m= 0$ (resp. $n=0$), we have the Weyl  vertex algebra  $M_{(n)} := M_{(0, 2n)}$, (resp. the Clifford vertex algebra $F_{(m)}:= M_{(m, 0)}$).   
 Clearly, we have the isomorphism:
$$  M_{(m|2n)} \cong F_{(m)} \otimes M_{(n)}. $$

It was proved in \cite{AdP-2019, AW, RW} that the Weyl vertex algebra $M_{(n)}$ has a remarkable family of  irreducible positive energy modules, which we denote by  $\widetilde{U_n}({\bf  a})$,  whose top components
 are isomorphic to   $$ U_n({\bf a}) = x_1 ^{a_1} \cdots x_n ^{a_n} {\C}[ x_1, \cdots, x_n, x_1 ^{-1}, \cdots, x_{n} ^{-1}], \quad  \mbox{for} \  {\bf a} =(a_1, \cdots, a_n) \in ({\C} \setminus {\Z})^n.  $$

It was proved in  \cite[Section 6]{AMPP-20} that  the simple  affine vertex algebra
$ V_{-1/2} (spo(2r \vert s))$ is realized as  a fixed point subalgebra  $M_{(s \vert 2 r)}^+$ of $M_{(s \vert 2 r)}$. In particular, for $\g=spo(2|3)$, we have $$V_{-1/2} (\g) = M_{(3 \vert 2 )}^+.$$ 
From  Proposition \ref{singprime2} \st{below} it follows that  $$V_{-m/4} (\g) = \frac{ V^{-m/4} (\g) }{ V^{-m/4} (\g)((x_{\theta_1})_{(-1)})^{m+1}\vac}. $$
This easily  implies that
$$ V_{-m/4 }(\g) \hookrightarrow  V_{-1/2} (\g)  ^{\otimes \tfrac{m}{2} } \hookrightarrow  M_{(3 \vert 2 )}^{\otimes \tfrac{m}{2}} \cong   M_{( \tfrac{3m}{2}  \vert m)}. $$
In this way, we get a non-zero vertex algebra homomorphism
$\Phi^{(m)} :  V_{-m/4 }(\g) \rightarrow   M_{( \frac{3m}{2}  \vert m)}$  for each $m\ge 4$ even.
Now using positive energy  modules  $\widetilde{U_m}({\bf  a})$ for $M_{(m)}$, we construct $M_{( \frac{3m}{2}  \vert m)}$--modules
$  F_{( \frac{3m}{2})  }  \otimes \widetilde{U_m}({\bf a})$, 
which by restriction become $V_{-m/4 }(\g)$--modules. Since their top components are not highest weight $\g$--modules,  we have constructed a family of non highest weight, positive energy  $V_{-m/4 }(\g)$--modules.
 Interestingly, these modules  are still  integrable for $\widehat \g^\natural$. We shall present details and irreducibility analysis in our forthcoming publications.
 %
 

 \subsection{Non-positive energy  modules over $W_k^{min}(\g)$}
 For a conformal vertex algebra $V$ let $\mathcal E(V)$ be the category of all (weak) $V$--modules,
 on which $L_0$ is diagonalizable, 
 and let  $\mathcal E^+(V)$ be the category of  positive energy modules in $\mathcal E(V)$. 
 Set $\mathcal E_k = \mathcal E (W_k^{min}(\g)) $, $\mathcal E_k ^+ = \mathcal E ^+(W_k^{min}(\g)) $, where $k$ is from the unitary range for $\g$.

 A  $W_k^{min}(\g)$--module  $M$ is called a  weight $W_k^{min}(\g)$--module if it is a weight  module for $\widehat{\g}^{\natural}$. 
  Let $\mathcal E_k^{fin} $ be the subcategory of  $\mathcal E_k  $   consisting of weight modules with finite multiplicities.

Theorem \ref{ipehw} shows that each irreducible module in $\mathcal E_k ^+$ belongs to the category $\mathcal{E}_k ^{fin}$.
One very interesting question is to see if there are   weight representations of  $W^{min}_k(\g)$ outside of the category $\mathcal E_k ^+$.   We believe that there are no such modules in $\mathcal{E}_k ^{fin}$:
\begin{conj} The irreducible highest weight $\Ws$-modules are all the  irreducible   representations in the category $\mathcal{E}_k ^{fin}$.
\end{conj}
However, weight modules   which  are not in  $\mathcal E_k ^+$ do exist. In order to see this, we use 
Kac-Wakimoto free field realization of $W_{min}^k(\g)$ \cite[Theorem 5.2]{KW1}, which gives a vertex algebra homomorphism
$$\Psi:W^k_{\min}(\g)\to \mathcal H \otimes \mathbb V^k(\g^\natural)\otimes F(\g_{1/2}),$$
where
\begin{itemize}
\item $\mathcal H$ is the Heisenberg vertex algebra of rank 1 generated by a field $a$ such that $[a_{\lambda} a] = \lambda$;
\item $\mathbb V^k(\g^{\natural}) = \bigotimes_{i\in S} V^{M_i(k)+\chi_i}(\g^{\natural}_i)$;
\item $F(\g_{1/2})$ is a fermionic vertex algebra.
\end{itemize}

 \begin{lemma}  \label{exlem-f-1} There is a embedding of  vertex algebras $\Phi: \bigotimes_{i\in S}V_{-\chi_i} (\g_i^{\natural}) \rightarrow   F(\g_{1/2})$ uniquely determined by
\begin{equation}\label{b} b \mapsto  \tfrac{1}{2} \sum_{ \alpha \in A} : \Phi ^{\alpha} \Phi_{[u_{\alpha}, b]}:,\quad b\in \g^\natural.
\end{equation}
\end{lemma}
\begin{proof} We start with Kac-Wakimoto free field realization \cite[Theorem 5.1]{KW1} $$\Psi':W^k_{\min}(\g)\to \mathcal V=V^{1}(\C a)\otimes V^{\alpha_k}(\g^\natural)\otimes F(\g_{1/2}),$$  where
$$\a_k(a, b)=\d_{i,j} (M_i(k)+\chi_i)(a|b)_i^\natural \text{ for $a\in \g^\natural_i, b\in \g^\natural_j,\ i,j\ge 0$.}
$$
Choose $\mu\in \C$ such that the $\mathcal V$-invariant Hermitian form on $M(\mu)\otimes V^{\alpha_k}(\g^\natural)\otimes F(\g_{1/2})$ is invariant for the action of $W^k_{\min}(\g)$. Since $ \Psi'(J^{\{b\}})\in V^{\alpha_k}(\g^\natural)\otimes F(\g_{1/2})$, it follows that  the map $J^{\{b\}}\mapsto \Psi'(J^{\{b\}})_{(-1)}(v_\mu\otimes \vac\otimes \vac),\,b\in\g_i^\natural$, extends to a map 
\begin{equation}\label{a}\bigotimes_{i\in S}V^{M_i(k)}(\g^\natural_i)\to  \C v_\mu\otimes  V_{\alpha_k}(\g^\natural)\otimes F(\g_{1/2})\end{equation}
that, identifying $\C v_\mu\otimes V_{\alpha_k}(\g^\natural)\otimes F(\g_{1/2})$ with  $V_{\alpha_k}(\g^\natural)\otimes F(\g_{1/2})$ in the obvious way, is a vertex algebra homomorphism.
Now plug $\a_i(k)=0$, i.e. $M_i(k)=-\chi_i$, into \eqref{a}. We get the homomorphism  $\Phi':\bigotimes_iV^{-\chi_i} (\g_i^{\natural}) \rightarrow   F(\g_{1/2})$ explicitly given by \eqref{b}. Since $ F(\g_{1/2})$  is unitary, it is completely reducible, hence the image of $\Phi'$ is simple.
Therefore $\Phi'$ descends to $\Phi$.

\end{proof}

The map $\Psi'$ induces a map $\Psi: W^k_{\min}(\g)\to V^{1}(\C a)\otimes V_{\alpha_k}(\g^\natural)\otimes F(\g_{1/2})$.  We  first prove that  the image of $\Psi$ is simple.

\begin{proposition}\label{p3} Assume that  $ M_i(k) + \chi_i  \in {\Z}_{\ge 0}$ for all $i$. Then $\Psi(W^k_{\min}(\g))=W_k^{\min}(\g)$. 
\end{proposition}

\begin{proof} As shown in the proof of Theorem \ref{C}, the vectors $v_i$ generate the maximal ideal in $\Wu$ (although $v_1$ is not a singular vector when $\g=spo(2|3)$). Hence 
it suffices to check that  in  $V^{1}(\C a)\otimes V_{\alpha_k}(\g^\natural)\otimes F(\g_{1/2})$ we have:
$$
(J^{\{x_{\theta_i}\}}_{(-1)})^{M_i(k)+1}\vac =0 \quad i \ge 1. $$
 From Kac-Wakimoto  free field realization and Lemma \ref{exlem-f-1}  we get that the map 
$ b \mapsto b + \tfrac{1}{2} \sum_{ \alpha \in A} : \Phi ^{\alpha} \Phi_{[u_{\alpha}, b]}:$
induces a homomorphism $$\Theta:=\Phi\circ \Psi_{|V^{M_i(k) } }:  V^{M_i(k) } (\g ^{\natural} _i) \rightarrow  V_{M_i(k) + \chi_i }(\g^{\natural} _i) \otimes V_{-\chi_i} (\g ^{\natural} _i) \subset   V_{\alpha_k}(\g^\natural)\otimes F(\g_{1/2}).$$
Since $ M_i(k) + \chi_i  \in {\Z}_{\ge 0}$,  the vertex algebra $V_{\alpha_k}(\g^\natural)\otimes F(\g_{1/2})$ is unitary, hence it  is completely reducible, so $Im\,\Theta$ is simple and in turn $(J^{\{x_{\theta_i}\}}_{(-1)})^{M_i(k)+1}\vac =0$.
\end{proof}
We have therefore proved the following result.

\begin{theorem} \label{free-f} If $k$ is non-collapsing and lies in  the unitary range, then $\Psi$ induces a non-trivial homomorphism of vertex algebras
$$\widetilde \Psi:W_k^{\min}(\g)\to \mathcal H \otimes \mathbb V_k(\g^\natural)\otimes F(\g_{1/2}),$$
where   $\mathbb V_k(\g^{\natural}) = \bigotimes_{i\in S} V_{M_i(k)+\chi_i}(\g^{\natural}_i)$.
\end{theorem}

 Since $\mathbb V_k(\g^{\natural}) $ and $F(\g_{1/2})$ are rational vertex algebras, they don't have non-positive energy modules.
 But the Heisenberg vertex algebra $\mathcal H$ is not rational and it has very rich representation theory:
 \begin{itemize}
 \item[(fin)] The category of $\mathcal H$--modules with finite-dimensional weight spaces $\mathcal E ^{fin} (\mathcal H) $  is semi-simple and every irreducible  $\mathcal H$--module in that category is of highest weight. The category $\mathcal E ^{fin} (\mathcal H) $  concides with the  category of modules in $\mathcal E^+ (\mathcal H)$ of finite length.
 \item[(inf)]  The category $\mathcal E(\mathcal  H)$ admits irreducible $\mathcal H$--modules with infinite-dimensional weight spaces \cite{GFM}.  
 \end{itemize}
 
\begin{remark}\label{later}
 Using Theorem \ref{free-f} and the representation theory of the Heisenberg vertex algebra $\mathcal H$ one  can prove the following facts:
assume that $E_1$ is an  
 irreducible 
 module in   $\mathcal E(\mathcal H)$. Let $E_2$ be a $\mathbb V_k(\g^{\natural}) $-module. Then  we have:
 
 \begin{itemize}
 \item[(1)] $E_1 \otimes E_2 \otimes F(\g_{1/2})$ lies in $\mathcal E_k$. It lies  in the category $\mathcal E^+ _k$ if and only if $E_1$ lies  in the category $\mathcal E^+(\mathcal H)$.
 
\item[(2)]  If $E_1$ is a weight   $\mathcal H$-module with infinite-dimensional weight spaces    then   $E_1 \otimes E_2 \otimes F(\g_{1/2})$ is a non-positive energy weight   $W_k^{min}(\g)$--module with infinite-dimensional weight spaces.
 
  \item[(3)] If $E_1$ is a non-weight  $\mathcal H$-module,     then   $E_1 \otimes E_2 \otimes F(\g_{1/2})$ is a non-weight  $W_k^{min}(\g)$-module. In particular, we can construct analogs of non-weight  $V_k(\g)$--modules from Remark \ref{non-wh-1}. 
  \end{itemize}
  
 \end{remark}
 \vskip5pt
 \noindent  Further details on applications of  Theorem \ref{free-f} and Remark \ref{later} will appear elsewhere.
 \subsection*{Acknowledgements}
The author thanks the referees for many constructive suggestions. The authors are grateful to  Maria Gorelik for correspondence. 
D.A.\ is partially supported by the
QuantiXLie Centre of Excellence, a project cofinanced
by the Croatian Government and European Union
through the European Regional Development Fund - the
Competitiveness and Cohesion Operational Programme
(KK.01.1.1.01.0004).
V.K. is partially supported by the Simons Collaboration grant.

 \section*{Declarations}
\noindent{\bf Competing Interests.} The authors have no competing interests to declare that are relevant to the content of this article.

\noindent {\bf Data Availability Statement.} Data sharing not applicable to this article as no datasets were generated or analysed during the current study.


  \vskip20pt
  \footnotesize{
  \noindent{\bf D.A.}:  Department of Mathematics, Faculty of Science, University of Zagreb, Bijeni\v{c}ka 30, 10 000 Zagreb, Croatia;
{\tt adamovic@math.hr}
  
\noindent{\bf V.K.}: Department of Mathematics, MIT, 77
Mass. Ave, Cambridge, MA 02139;\newline
{\tt kac@math.mit.edu}

\noindent{\bf P.MF.}: Politecnico di Milano, Polo regionale di Como,
Via Anzani 42, 22100 Como,
Italy; 

\noindent {\tt pierluigi.moseneder@polimi.it}

\noindent{\bf P.P.}: Dipartimento di Matematica, Sapienza Universit\`a di Roma, P.le A. Moro 2,
00185, Roma, Italy; {\tt papi@mat.uniroma1.it}}
\end{document}